\documentclass[12pt,a4paper]{amsart}
\usepackage{amsfonts, amsmath, amssymb, amscd, amsthm, array}

\usepackage{multicol}
\usepackage{subfig}
\usepackage{color}
\usepackage[all]{xy}
\usepackage{hyperref,url}

\makeatletter
\def\blfootnote{\xdef\@thefnmark{}\@footnotetext}
\makeatother

\newtheorem{thm}{Theorem}[section]
\newtheorem*{thmA}{Theorem A}
\newtheorem*{thmB}{Theorem B}
\newtheorem*{thmC}{Theorem C}
\newtheorem{cor}[thm]{Corollary}
\newtheorem{lemma}[thm]{Lemma}
\newtheorem{prop}[thm]{Proposition}

\theoremstyle{definition}
\newtheorem{df}[thm]{Definition}
\theoremstyle{remark}
\newtheorem{rem}[thm]{Remark}
\newtheorem{Ex}[thm]{Example}

\newcommand{\G}{\mathfrak{G}}
\newcommand{\ga}{\Gamma}
\newcommand{\GG}{\Gamma\mathfrak{G}}
\newcommand{\ISG}{\mathcal{I}\mathcal{S}\mathfrak{G}}

\newcommand{\N}{\mathbb{N}}
\newcommand{\Z}{{\mathbb Z}}
\newcommand{\D}{{\mathbb D}}
\newcommand{\vls}{{\rm lv}}
\newcommand{\vfs}{{\rm fv}}
\newcommand{\pc}{{\rm Pc}}
\newcommand{\nafs}{non-abelian free subgroup}
\newcommand{\ff}{$\mathbb{F}_2$}
\newcommand{\graphx}{\Gamma}

\newcommand{\graphz}{\Delta}

\newcommand{\No}{{\rm N}}

\DeclareMathOperator{\link}{link}
\DeclareMathOperator{\supp}{supp}
\DeclareMathOperator{\esupp}{esupp}
\DeclareMathOperator{\axis}{axis}
\DeclareMathOperator{\rank}{rank}
\DeclareMathOperator{\frank}{frank}

\def\coloneqq{\mathrel{\mathop\mathchar"303A}\mkern-1.2mu=}
\newcommand{\gen}[1]{\left\langle#1\right\rangle}

\begin{document}

\title[Tits alternatives for graph products]{Tits alternatives for graph products}
\address{School of Mathematics,
University of Southampton, Highfield, Southampton, SO17 1BJ, United Kingdom.}
\author{Yago Antol\'{i}n}
\email[Yago Antol\'{i}n]{yago.anpi@gmail.com}

\author{Ashot Minasyan}
\email[Ashot Minasyan]{aminasyan@gmail.com}

\thanks{This work was supported by the EPSRC grant EP/H032428/1. The first author was also supported by MCI (Spain) through project MTM2008--1550.}

\date{\footnotesize\today}

\begin{abstract} We discuss various types of Tits Alternative for subgroups of graph products of groups,
and prove that, under some natural conditions, a graph product of groups satisfies a given form of Tits Alternative if and only if each vertex group satisfies this alternative.
As a corollary, we show that every finitely generated subgroup of a graph product of virtually solvable groups is either virtually solvable or large.
As another corollary, we prove that every non-abelian subgroup of a right angled Artin group has an epimorphism onto the free group of rank $2$.
In the course of the paper we develop the theory of parabolic subgroups, which allows to describe the structure of subgroups of graph products
that contain no non-abelian free subgroups. We also obtain a number of results regarding the stability of some group properties under taking graph products.
\end{abstract}


\maketitle

\setcounter{tocdepth}{1}
\tableofcontents


\section{Introduction}

In 1972 J. Tits \cite{Tits} proved that a finitely generated linear group either is virtually solvable, or contains a copy of
the free group  $\mathbb{F}_2$ of rank $2$.  Nowadays such a dichotomy is called the \textit{Tits Alternative}.
This alternative is a powerful result with important consequences. It was used by M. Gromov in his proof of the
famous ``\textit{polynomial growth}'' theorem \cite{Gromov-polyn_gr}; it also shows that linear groups
satisfy \textit{von Neumann conjecture}: every linear group is either amenable or contains a \nafs.

The Tits Alternative can be naturally modified by substituting ``finitely generated'',  ``virtually solvable'' and ``contains a \nafs'' with other conditions of a similar form.
An example of a result of this type is the theorem of G. Noskov and E. Vinberg \cite{NoskovVinberg} claiming that every subgroup of a
finitely generated Coxeter group is either virtually abelian or large.
Recall that a group is said to be \emph{large} if it has a finite index subgroup which maps onto a non-abelian free group. A large group always contains a \nafs, but not vice-versa.

Let us now formally describe the possible versions of the Tits alternative that we are going to consider.
Let $\mathcal{I}$ be a collection of cardinals; a group $G$ is said to be \textit{$\mathcal{I}$-generated},
if there is a generating set $X$ of $G$ and $\lambda \in \mathcal{I}$ such that $|X|\le \lambda.$

\begin{df}
Suppose that $\mathcal I$ is a collection of cardinals, $\mathcal C$ is a class of groups and  $G$ is a group.
We will say that $G$ \textit{satisfies the  Tits Alternative relative to} $(\mathcal{I},\mathcal C)$ if for any $\mathcal{I}$-generated subgroup
$H \leqslant G$ either $H \in \mathcal{C}$ or $H$ contains a non-abelian free subgroup.

The group $G$ \textit{satisfies the Strong Tits Alternative relative to}  $(\mathcal{I},\mathcal C)$ if for any $\mathcal{I}$-generated subgroup
$H \leqslant G$ either $H \in \mathcal{C}$ or $H$ is large.
\end{df}

In this terminology the theorem of Tits \cite{Tits}, mentioned above, tells us that
linear groups satisfy the Tits Alternative relative to $(\mathcal{I}_f,\mathcal{C}_{vsol}),$
where $\mathcal{I}_f$ is the collection of all finite cardinals and $\mathcal{C}_{vsol}$ is the class of virtually solvable groups.
For the class $\mathcal{C}_{vab}$, of virtually abelian groups, the Tits Alternative relative to $(\mathcal{I}_f,\mathcal{C}_{vab})$ is known to hold
in any word hyperbolic group \cite{Gromov}, $Out(\mathbb{F}_n)$ \cite{BFH}, where $\mathbb{F}_n$ is the free group
of rank $n$, and in any group acting freely and properly on a finite dimensional CAT($0$) cubical complex \cite{Sageev-Wise}.

Let $\mathcal{I}_\omega$ be the collection of all countable cardinals. The result of G. Noskov and E. Vinberg \cite{NoskovVinberg} in this language becomes:
finitely generated Coxeter groups satisfy the Strong Tits Alternative relative to $(\mathcal{I}_\omega,\mathcal{C}_{vab})$.

The goal of this paper is to prove that many forms of Tits Alternative are stable under graph products.
Let $\graphx$ be a simplicial graph and suppose that $\mathfrak{G}=\{G_v \mid v\in V\graphx\}$ is a collection of groups (called \textit{vertex groups}).
The \emph{graph product} $\graphx \mathfrak{G}$, of this collection of groups with respect to $\graphx$,
is the group obtained from the free product of the $G_v$, $v \in V\graphx$, by adding the relations
$$[g_v, g_u]=1  \text{ for all }  g_v\in G_v,\, g_u\in G_u \text{ such that $\{v,u\}$ is an edge of } \graphx.$$

The graph product of groups is a natural group-theoretic construction generalizing free products (when $\graphx$ has no edges) and direct products (when $\graphx$ is a complete graph) of
groups $G_v$, $v \in V\graphx$. Graph products were first introduced and studied by E. Green in her Ph.D. thesis \cite{Green}. Further properties of graph products have been investigated by
S. Hermiller and J. Meier in \cite{H-M} and by T. Hsu and D. Wise in \cite{HsuWise}.

Basic examples of graph products are \textit{right angled Artin groups}, also called \textit{graph groups} (when all vertex groups are infinite cyclic), and \textit{right angled Coxeter groups}
(when all vertex groups are cyclic of order $2$).

Throughout this paper $\Z$ will denote the group of integers under addition and
$\mathbb{D}_\infty \cong \Z/2\Z*\Z/2\Z$ will denote the infinite dihedral group.
We will be interested in  collections of cardinals $\mathcal{I}$ and classes of groups $\mathcal C$ satisfying the following properties:
\begin{itemize}
  \item[(P$0$)] (\textit{closed under isomorphisms}) if $A,B$ are groups, $A \in \mathcal{C}$ and $A \cong B$ then $B \in \mathcal{C}$;
  \item[(P$1$)] (\textit{closed under $\mathcal{I}$-generated subgroups}) if $A \in \mathcal{C}$ and $B \leqslant A$ is an $\mathcal{I}$-generated subgroup, then $B \in \mathcal{C}$; 	
  \item[(P$2$)] (\textit{closed under direct products of $\mathcal{I}$-generated groups}) if $A,B \in \mathcal{C}$ are $\mathcal{I}$-generated then  $A\times B \in \mathcal{C}$;
  \item[(P$3$)] (\textit{contains the infinite cyclic group}) $\mathbb{Z} \in \mathcal{C}$;
  \item[(P$4$)] (\textit{contains the infinite dihedral group}) if $\Z/2\Z \in \mathcal{C}$ then $\D_\infty \in \mathcal{C}$.
\end{itemize}

\begin{thmA} Suppose that $\mathcal{I}$ is a collection of cardinals and $\mathcal{C}$ is a class of groups enjoying {\rm (P$0$)--(P$4$)}. Let
$\ga$ be a finite graph and let $\mathfrak{G}=\{G_v \mid v \in V\}$ be a family of groups.
Then the graph product $G=\GG$ satisfies the Tits Alternative relative to $(\mathcal{I},\mathcal{C})$ provided each vertex group $G_v$, $v \in V$, satisfies this alternative.
\end{thmA}

To establish the strong version of Tits alternative we will need one more condition on $\mathcal{I}$ and $\mathcal C$:
\begin{itemize}
  \item[(P$5$)] (\textit{$\mathcal{I}$-locally profi}) if $A \in \mathcal{C}$ is  non-trivial and $\mathcal{I}$-generated then $A$ possesses a proper finite index subgroup.
\end{itemize}

\begin{thmB} Let $\mathcal{I}$ be a collection of cardinals and let $\mathcal C$ be a class of groups enjoying the properties {\rm (P$0$)-(P$5$)},
such that $\mathcal I$ contains all finite cardinals or at least one infinite cardinal.  Suppose that $\ga$ is a finite
graph and $\mathfrak{G}=\{G_v \mid v\in V\graphx\}$ is a family of groups.
Then the graph product $G=\GG$ satisfies the Strong Tits Alternative relative to $(\mathcal{I},\mathcal{C})$ provided  each vertex group $G_v$, $v \in V\graphx$,
satisfies this alternative.
\end{thmB}

Examples of classes of groups with properties {\rm (P$0$)-(P$5$)}, for $\mathcal{I}=\mathcal{I}_f$,
are the classes consisting of virtually abelian groups, virtually nilpotent groups, (virtually) polycyclic groups, (virtually) solvable groups and, more generally, elementary amenable groups.
It is easy to see that all of these properties are necessary. For example, if groups from $\mathcal C$ contain no free subgroups and
$A\in \mathcal{C}$ is a non-trivial group without proper finite index subgroups then  $A*A$ will possess no non-trivial finite quotients.
It follows that $A*A$ cannot be large; on the other hand, $A*A \notin \mathcal{C}$ as it contains a copy of \ff.
Thus we see that property (P$5$) is necessary for the claim of Theorem B.

For any $n \in \mathbb{N}$, let  $\mathcal{C}_{sol-n}$ be the class of all solvable groups of derived length at most $n$.
Denote by $\mathcal{C}_{vsol-n}$ the class of all groups
that are virtually in $\mathcal{C}_{sol-n}$. It is easy to see that the pairs $(\mathcal{I}_f,\mathcal{C}_{sol-m})$ and
$(\mathcal{I}_f,\mathcal{C}_{vsol-n})$  enjoy the properties {\rm (P$0$)-(P$5$)}, for all $m,n \in \N$ with $m \ge 2$
(when $m=1$ the class $\mathcal{C}_{sol-1}$, of abelian groups, does not satisfy (P$4$)).
Hence applying Theorem B we achieve

\begin{cor}\label{cor:vsol-n} Suppose that $\mathcal C=\mathcal{C}_{sol-m}$ for some $m \ge 2$ or $\mathcal{C}=\mathcal{C}_{vsol-n}$ for some $n \ge 1$.
Let $G$ be a graph product of groups from  $\mathcal C$.
Then any finitely generated subgroup $H$ of $G$ either is large or belongs to $\mathcal{C}$.
\end{cor}

Observe that the graph product $G$ could be taken over an infinite graph in the above corollary because a finitely generated subgroup is always contained
in a ``sub-graph product'' with only finitely many vertices (see Remark \ref{rem:fg_sbgp-fin_supp}).
It is also worth noting that the group $\mathbb Q$, of rational numbers, can be embedded into a finitely generated group from $\mathcal{C}_{sol-3}$ (this follows from the fact that there are
finitely generated center-by-metabelian groups whose center is the free abelian group of countably infinite rank -- see \cite{P.Hall-1};
taking a quotient of such a group by an appropriate central
subgroup produces an embedding of an arbitrary countable abelian group into the center of a finitely generated group from $\mathcal{C}_{sol-3}$).
Consequently, the free square of a finitely
generated group from $\mathcal{C}_{sol-3}$ can contain $\mathbb{Q}*\mathbb{Q}$, which is neither solvable nor large.
Thus the assumption that $H$ is finitely generated in Corollary \ref{cor:vsol-n} is essential.

Applying Theorem B to finitely generated right angled Coxeter groups, we see that these groups satisfy the Strong Tits Alternative relative to
$(\mathcal{I}_\omega,\mathcal{C}_0)$, where $\mathcal{C}_0$ is the smallest class of groups containing $\D_\infty$, closed under isomorphisms, taking subgroups and direct products.
It is easy to see that $\mathcal{C}_0\subset \mathcal{C}_{vab} \cap \mathcal{C}_{sol-2} \cap \mathcal{C}_{polyc}$,
where $\mathcal{C}_{polyc}$ denotes the class of polycyclic groups; in particular all groups in $\mathcal{C}_0$ will be finitely generated.
A theorem of F. Haglund and D. Wise \cite{HaglundWise} states that every finitely generated Coxeter group is virtually a subgroup of some
finitely generated right angled Coxeter group. Therefore we recover the result of Noskov-Vinberg  \cite{NoskovVinberg}, mentioned above:

\begin{cor}\label{cor:coxeter} Let $G$ be a finitely generated Coxeter group.
If $H\leqslant G$ is an arbitrary subgroup then either $H$ is large or
$H$ is finitely generated and virtually abelian.
\end{cor}

Finally we would like to suggest the strongest possible (in our opinion) form of Tits alternative as follows.
\begin{df}
Suppose that $\mathcal I$ is a collection of cardinals, $\mathcal C$ is a class of groups and  $G$ is a group.
We will say that $G$ \textit{satisfies the Strongest Tits Alternative relative to} $(\mathcal{I},\mathcal C)$ if for any $\mathcal{I}$-generated subgroup
$H \leqslant G$ either $H \in \mathcal{C}$ or $H$ has an epimorphism onto the free group \ff{} of rank $2$.
\end{df}

The group $G= \gen{a,b,c \mid a^2b^2=c^2}$ is an example of a torsion-free large group (\cite{Baumslag-Pride}) which does not map onto \ff{} (\cite{Lyndon}).
Thus the Strongest Tits Alternative is indeed more restrictive than the Strong Tits Alternative.
The evident groups, satisfying the Strongest Tits Alternative relative to the class 
of torsion-free abelian groups (and arbitrary $\mathcal I$), are residually free groups.

For a collection of cardinals $\mathcal I$, we define a new collection of cardinals $\mathcal{I}^-$, by saying that a cardinal $\lambda$ belongs to
$\mathcal{I}^-$ if and only if $\lambda+1 \in \mathcal{I}$ (see Sub-section \ref{s:cardinals} for the definition of the addition of cardinals). For instance, if
$\mathcal{I}=\{2\}$ then $\mathcal{I}^-=\{1\}$; note also that $\mathcal{I}_f^-=\mathcal{I}_f$ and $\mathcal{I}_\omega^-=\mathcal{I}_\omega$.
In order to prove the strongest alternative, we need an additional property of the pair $(\mathcal{I},\mathcal{C})$:

\begin{itemize}
\item[(P$6$)] (\textit{$\mathcal{I}^-$-locally indicable}) if $A \in \mathcal{C}$ is  non-trivial and $\mathcal{I}^-$-generated then $A$ has an infinite cyclic quotient.
\end{itemize}

Evidently (P$6$) implies that every group in $\mathcal C$ is torsion-free
(provided $\mathcal{I}^-$ contains at least one non-zero cardinal). Basic examples of groups satisfying (P$6$) with $\mathcal{I}=\mathcal{I}_f$ are torsion-free nilpotent
groups (see \cite[5.2.20]{Robinson}).

\begin{thmC} Let $\mathcal{I}$ be a collection of cardinals and let $\mathcal C$ be a class of groups enjoying the properties {\rm (P$0$)-(P$3$)} and {\rm (P$6$)}.
Suppose that $\ga$ is a finite graph and $\mathfrak{G}=\{G_v \mid v\in V\graphx\}$ is a collection of groups.
Then the graph product $G=\GG$ satisfies the Strongest Tits Alternative relative to $(\mathcal{I},\mathcal{C})$
provided each vertex group $G_v$, $v \in V\graphx$, satisfies this alternative.
\end{thmC}

Consider any pair $(\mathcal{I},\mathcal C)$ satisfying the conditions of Theorem C, such that $\mathbb{F}_2 \notin \mathcal{C}$.
Then for an arbitrary non-trivial $\mathcal{I}^-$-generated group $A \in \mathcal{C}$, $G=A*\Z$ is $\mathcal{I}$-generated and
contains a copy of $\mathbb{F}_2$, hence  $G \notin \mathcal{C}$.
So, the Strongest Tits Alternative (relative to $(\mathcal{I},\mathcal{C})$)
for $G$ would imply that $G$ maps onto $\mathbb{F}_2$, hence the image of $A$ under
this homomorphism is non-trivial, and, so it possesses an epimorphism onto $\mathbb Z$. Thus the assumption of (P$6$) in Theorem C is indeed necessary.

For the first application of Theorem C, let us take $\mathcal{I}=\{2\}$ and let $\mathcal{C}$ be any class of torsion-free
groups enjoying the properties (P$0$)-(P$3$) (e.g., $\mathcal{C}$ could be the class of torsion-free amenable groups).
In this case $\mathcal{I}^-$-local indicability is equivalent to torsion-freeness, and
so (P$6$) holds automatically. Recalling that a $2$-generated group maps onto \ff{} if and only if it is isomorphic to \ff, we achieve

\begin{cor}\label{cor:2-gen}
Let $\mathcal{C}$ be a non-empty class of torsion-free groups, closed under isomorphisms, direct products and taking subgroups.
Suppose that $G$ is a graph  product with all vertex groups from $\mathcal C$. Then for any $2$-generated subgroup $H \leqslant G$, either $H \in \mathcal C$ or
$H \cong \mathbb{F}_2$.
\end{cor}

Corollary \ref{cor:2-gen} generalizes  a classical result of A. Baudisch \cite{Baud}, who proved that any
two non-commuting elements of a right angled Artin group generate a copy of $\mathbb{F}_2$.
In fact, we are now able to say much more about subgroups of right angled
Artin groups. Indeed, taking $\mathcal{I}=\mathcal{I}_\omega$ and $\mathcal C$ to be the class of finitely generated torsion-free abelian groups,
and applying Theorem~C we obtain the following result:

\begin{cor}\label{cor:RAAG}
Any subgroup of a finitely generated right angled Artin group is either free abelian of finite rank or maps onto $\mathbb{F}_2$.
\end{cor}

The finite generation assumption in the above statement is not very important. Indeed, using Corollary \ref{cor:RAAG} together with Remark  \ref{rem:fg_sbgp-fin_supp},
it is easy to show that any non-abelian subgroup of an arbitrary right angled Artin group maps onto \ff.
An easy consequence of Corollary \ref{cor:RAAG} is that for every non-cyclic subgroup $H$, of a right angled Artin group,
the first Betti number (i.e, the $\mathbb Q$-rank of the abelianization) is at least two.

As far as we know, Corollary \ref{cor:RAAG} gives the first non-trivial family of groups satisfying the Strongest Tits Alternative
(it is not hard to construct right angled Artin groups that are not residually free). The family of  subgroups of right angled Artin groups is very rich
and includes a lot of examples examples (see \cite{Agol,B-B,Crisp-Wiest,HaglundWise,Haglund-Wise-0,Liu,Wise-qch}). However,  Corollary \ref{cor:RAAG} can be used to show that many groups are not
embeddable into right angled Artin groups. More precisely, one can obtain information regarding solutions of an equation in a right angled Artin group
from the solutions of the same equation in free and free abelian groups. For example, we have

\begin{cor}\label{cor:abc} If three elements $a,b,c$ of a right angled Artin group $G$ satisfy $a^m=b^nc^p$ for $m,n,p \ge 2$, then these elements pairwise commute.
\end{cor}

The above fact is an immediate consequence of Corollary \ref{cor:RAAG} and the result of R. Lyndon and M. Sch\"utzenberger \cite{Lyn-Schutz}, who proved the same statement when $G$ is free.
Corollary \ref{cor:abc} generalizes a theorem of J. Crisp and B. Wiest \cite[Thm. 7]{Crisp-Wiest}, who established its claim in the case when $m=n=p=2$.

The paper is organized as follows. In Section \ref{sec:basic_prop} we recall some basic properties of graph products, and in Section \ref{S:parabolic} we develop the theory of parabolic subgroups of graph products.
The results of this section, together with Bass-Serre theory, allow us to prove, in Section \ref{S:structuretheorem}, the Structure Theorem for subgroups of graph products which
do not contain a copy of \ff{} (Theorem \ref{thm:structure}).
The Structure Theorem is one of the main results of this work, and it is intensively used in the proofs of Theorems A, B and C.
We apply the Structure Theorem to give a more detailed description of ``small'' subgroups of graph products (see Theorem \ref{thm:sharpA}), and Theorem A is
a consequence of this description.

In order to  prove Theorems B and C, we study the kernel of the canonical retraction of a graph product onto one of the vertex groups and prove that this kernel
is again a graph product (Theorem \ref{thm:kernel}).  This fact has a number of interesting applications, which are established in Section \ref{S:kernel}.
In Section \ref{S:proof} we establish a sufficient criterion for a subgroup of a graph product to be large or to map onto \ff{} (Theorem \ref{thm:main_crit}), and then we
prove Theorems B and C. Finally, in Section \ref{S:example} we construct two examples showing that  the claims of Corollaries \ref{cor:2-gen} and \ref{cor:RAAG}
are in a certain sense optimal. Example \ref{ex:2-gen_non_virt_cyc} shows that without the assumption of torsion-freeness, there is no control over
$2$-generated subgroups of graph products; and
Proposition \ref{prop:ex_fr_rank}  demonstrates that there is no connection between the rank
of a non-abelian subgroup of a right angled Artin group and the maximal rank of the free group onto which it can be mapped. The latter is based on
a Rips-type construction, which is presented in Proposition \ref{prop:Rips_constr}.


\section{Preliminaries}\label{sec:basic_prop}

\subsection{Cardinal numbers}\label{s:cardinals}
Our reference for cardinal numbers is \cite{Vaught}.
Cardinals are traditionally well-ordered by injection ($|X|\le |Y|$ if there is an injective map from $X$ to $Y$),
but, assuming the Axiom of Choice, this ordering is equivalent to the ordering defined via surjections, which we will use (see \cite[Proposition 2.1.8]{Vaught}).
Formally, if $\lambda$, $\kappa$ are two cardinals, and $X$, $Y$ are sets with $|X|=\lambda$, $|Y|=\kappa$, then $\lambda \le \kappa$ if and only if
there exists a surjective map $Y\twoheadrightarrow X.$ The addition of cardinals $\lambda$ and $\kappa$ is defined in a natural way: $\lambda+\kappa$ is the cardinal of the union
$X\cup Y$ of two disjoint sets $X$ and $Y$ with $|X|=\lambda$ and $|Y|=\kappa$. A well-known fact tells us that $\lambda+\kappa=\max\{\lambda,\kappa\}$
if at least one of the cardinals $\lambda$, $\kappa$ is infinite (see \cite[Corollary 7.6.2]{Vaught}).
If $m \in \N$ and $\lambda$ is a cardinality of a set $X$ then $m\lambda$ denotes the cardinality of the disjoint union of $m$ copies of $X$; thus if $\lambda$ is an infinite cardinal
then $m\lambda=\lambda$ for any $m\in \N$.

Let $\mathcal{I}$ be a collection of cardinals. We say that a group $G$  is $\mathcal{I}$-generated, if there is a generating set $X$ of $G$ of
cardinality $\lambda$, and a cardinal $\kappa\in\mathcal{I}$ such that $\lambda \le \kappa$.
Remark that the image of an $\mathcal{I}$-generated group under a homomorphism is again $\mathcal{I}$-generated. It also follows from the definition that if $\mathcal J$ is another
collection of cardinals, such that for every $\lambda \in \mathcal{J}$ there is $\kappa \in \mathcal{I}$ with $\lambda \le \kappa$, then any $\mathcal J$-generated
group is also $\mathcal I$-generated. We will say that a collection of cardinals $\mathcal I$
is \textit{ample} if it contains all finite cardinals or if it contains at least one infinite cardinal.

\begin{lemma}\label{lem:card_of_fi} Let $\mathcal I$ be an ample collection of cardinals. If $G$ is an $\mathcal I$-generated group and
$K \leqslant G$ is a subgroup of finite index then $K$ is also $\mathcal I$-generated.
\end{lemma}

\begin{proof} By the assumptions, $G=\bigsqcup_{i=1}^k Kt_i$ for some $t_i \in G$, where $k\coloneqq |G:K|$.
Choose a generating set $X$ of $G$ with $|X| \le \lambda$ for some $\lambda \in \mathcal I$.
By a theorem of Reidemeister and Schreier (see \cite[6.1.8]{Robinson}),  $K$ can be generated by a subset $Y$ of the set $TXT^{-1}\subset G$, where $T\coloneqq \{t_1,\dots,t_k\}$.
Consequently $|Y| \le k^2 |X|$. If $X$ is finite then $Y$ is also finite, and hence $K$ is $\mathcal I$-generated by the assumptions on $\mathcal I$.
On the other hand, if $X$ is infinite, then $\lambda$ is an infinite cardinal, and so $|Y|\le k^2 \lambda =\lambda \in \mathcal{I}$, as required.
\end{proof}

\subsection{Graph products}
Let $\Gamma$ be a graph without loops or multiple edges. We will use $V\Gamma$ and $E\Gamma$ to denote the set of vertices and the set of edges of $\Gamma$ respectively.
An edge can be considered as a $2$-element subset of $V\Gamma$. A path in $\Gamma$ of length $n$ from $u$ to $v$ is a sequence of vertices $(u=v_0,v_1,\dots, v_n=v)$ where $\{v_{i-1},v_{i}\}\in E\Gamma$,
for $i=1,\dots, n$.

For any subset $A\subseteq V\Gamma$, by $\Gamma_A$ 
 we will denote the \textit{full subgraph} of $\Gamma$  with vertex set $A.$ That is,
$V\Gamma_A=A$ and $\{a_1,a_2\}\in E\Gamma_A$ if and only if $\{a_1,a_2\}\in E\Gamma$.

The \emph{link} $\link_{\Gamma}(v)$, of a vertex $v \in V\Gamma$, is the subset of vertices adjacent to $v$ (excluding the $v$ itself); in other words,
$\link_{\Gamma}(v)\coloneqq \{u\in V\Gamma \mid \{v,u\}\in E\Gamma \}.$ For a subset $A \subseteq V\Gamma$, we define $\link_\ga(A)\coloneqq \bigcap_{v \in A} \link_\ga(v)$.

Let $\mathfrak{G}=\{G_v\mid v\in V\Gamma\}$ be a family of groups and let $G\coloneqq \Gamma \mathfrak{G}$ be the corresponding graph product.
Any element $g\in G$  may be represented as a \emph{word} $W \equiv (g_1,g_2,\dots, g_n)$ 
 where each $g_i$, called a \emph{syllable} of $W$, is an element of some $G_v$ and $g=g_1g_2\dots g_n$.
The number of syllables is the length of the word. For $1 \le i < j \le n$, we will say that the syllables $g_i$ and $g_j$ \textit{can be
joined together} if $g_i,g_j\in G_v$ for some $v \in V\Gamma$ and  $g_k\in G_{u_k}$, $u_k\in \link_\Gamma (v)$ for $k=i+1,\dots, j-1$. In this case, in $\Gamma\mathfrak{G}$ the word $W$
represents the same element as the words $(g_1,\dots,g_{i-1},g_{i}g_j,g_{i+1},\dots,g_{j-1},g_{j+1},\dots,g_n)$ and $(g_1,\dots,g_{i-1},g_{i+1},\dots,g_{j-1},g_{i}g_j,g_{j+1},\dots,g_n)$,
whose lengths are strictly smaller than the length of $W$.

A word $W \equiv (g_1,g_2,\dots, g_n)$ is \emph{reduced} either if it is empty or 
if $g_i\neq 1$ for all $i$,  and no two distinct syllables of $W$ can be joined together.

We define the following transformations for the word $W$.
\begin{enumerate}
\item[(T1).] Remove a syllable  $g_i$ if $g_i=1$ in $G$.
\item[(T2).] Replace two consecutive syllables $g_i$ and $g_{i+1}$ in the same vertex group $G_v$ with the single syllable $g_ig_{i+1} \in G_v$.
\item[(T3).] (\textit{Syllable shuffling}) For consecutive syllables $g_i\in G_u,$ $g_{i+1}\in G_v$ with $\{u,v\}$ an edge of $\Gamma,$ interchange $g_i$ and $g_{i+1}.$
\end{enumerate}

Note that transformation (T3)  preserves the length of the word, and transformations (T1),(T2) decrease it by $1$.
Evidently starting with some word $W$ and applying finitely many transformations (T1)-(T3) we can obtain a reduced word $W'$, representing the same element of the group $\Gamma \mathfrak{G}$.

The following theorem was first proved by E. Green \cite[Thm. 3.9]{Green} in her thesis. 
\begin{thm}[The Normal Form Theorem]\label{thm:norm_form}
Every element $g\in \Gamma \mathfrak{G} $ can be represented by a reduced word. 
Moreover, if two reduced words represent the same element of the group, then one can be obtained from the other after applying a finite sequence of syllable shuffling.
In particular, the length of a reduced word is minimal among all words representing $g$, and a reduced word represents the identity if and only if it is the empty word. 
\end{thm}

Let $g\in G$ and $W\equiv(g_1,\dots,g_n)$ be a reduced word representing $g.$ We define the \emph{length}  of $g$ in $G$ to be $|g|_\Gamma=n$ and the $\emph{support}$ of $g$  in $G$ to be
$\supp_\Gamma(g)=\{v\in V\Gamma \mid \exists~i\in \{1,\dots,n\} \mbox{ such that }  g_i\in G_v-\{1\}\}.$ By the Normal Form Theorem, the length and support of an
element are well defined. For a subset $X \subseteq \Gamma \mathfrak{G}$,
the \textit{support of} $X$, $\supp_\Gamma (X)$ will be defined by $\supp_\Gamma (X)=\bigcup_{x \in X} \supp_\Gamma(x) \subseteq V\Gamma$. Observe that for a subgroup $H =\gen{X}\leqslant \Gamma \mathfrak{G}$
one has $\supp_\Gamma (H)=\supp_\Gamma (X)$. In particular,  if $H$ is finitely generated then $\supp_\Gamma (H)$ is a finite subset of $V\Gamma$.

For every $g \in G$, Theorem~\ref{thm:norm_form} also allows to define the set $\vfs_\ga(g)\subseteq V\ga$, consisting of all $v \in \supp_\ga(g)$
such that some reduced word $W$, representing $g$, starts with a syllable from $G_v$.
Similarly, $\vls_\ga(g) \subseteq V\ga$ will denote the subset of all $u \in V$ such that some reduced word for $g$ ends with a syllable from $G_u$.
Evidently $\vfs_\ga(g^{-1})=\vls_\ga(g)$.


\section{Parabolic subgroups of graph products}\label{S:parabolic}
Throughout this section $\ga$ will denote a simplicial graph, $V\coloneqq V\ga$, $\G=\{G_v \mid v \in V\}$ will be a family of groups and $G=\GG$ will be the corresponding graph product.
For any subset $A\subseteq V\Gamma,$ the subgroup $G_A\leqslant G$, generated by $\{G_v \mid v\in A\}$, is called a \textit{full subgroup}; according to a standard convention,
$G_\emptyset=\{1\}$.  By the Normal Form Theorem, $G_A$ is the graph product
of the groups $\{G_v \mid v \in A\}$ with respect to the graph $\Gamma_A$. It is also easy to see that there is a \textit{canonical retraction} $\rho_A\colon G \to G_A$ defined (on the generators of $G$) by
$\rho_A(g)\coloneqq g$ for each $g \in G_v$ with $v \in A$, and $\rho_A(h)\coloneqq 1$ for each $h \in G_u$ with $u \in V\Gamma - A$.

\begin{rem}\label{rem:fg_sbgp-fin_supp} If $H\leqslant G$ is a finitely generated subgroup then $H$ is contained in the
full subgroup $G_A$, of $G$, where $A:=\supp_\ga(H)$ is a finite subset of $V\ga$,
and thus $G_A$ is a graph products of the family $\{G_v \mid v \in A\}$ over a finite graph $\ga_A$.
\end{rem}

For any $v \in V$ the group $G=\Gamma\mathfrak{G}$ naturally splits as a free amalgamated product: $G=G_A*_{G_C}G_B$, where $C=\link_\ga(v)$, $B=\{v\}\cup\link_\ga(v)$ and $A=V-\{v\}$
(cf. \cite[Lemma 3.20]{Green}).

The goal of this section is to develop the theory of \textit{parabolic subgroups} (i.e., conjugates of full subgroups) of graph products.
Some of the statements in this section are similar to those already known about parabolic subgroups of Coxeter groups (see \cite[Sections 2,3]{Krammer}) and about
parabolic subgroups of right angled Artin groups (see \cite[Section 2]{D-K-R}).


\begin{lemma} \label{lem:sm_aux} Suppose that $T \subseteq V:=V\Gamma$ and $g,x,y$ are some elements of $G=\GG$ such that $gxg^{-1}=y$, $\supp(y) \subseteq T$ and $\vls_\ga(g) \cap \supp_\ga(x)=\emptyset$.
Then $g$ can be represented by a reduced word $(h_1,\dots,h_r,h_{r+1},\dots,h_n)$, where $h_1,\dots,h_r \in G_T$ and $h_{r+1},\dots,h_n \in G_S$, with $S\coloneqq \link_\ga(\supp_\ga(x)) \subseteq V$.
\end{lemma}

\begin{proof} Choose some reduced words $(g_1,\dots,g_n)$, $(x_1,\dots,x_k)$ and $(y_1,\dots,y_l)$ representing the elements $g$, $x$, and $y$ in $G$ respectively.
The statement will be proved by induction on $n=|g|_\ga$. If $n=0$ then the claim clearly holds, so assume that $n \ge 1$ and the claim has already been established for all $g \in G$, satisfying the assumptions of the lemma,
with $|g|_\ga<n$.

If $\supp_\ga(g) \subseteq \supp_\ga(y) \subseteq T$, there is nothing to prove. Otherwise, the equality $gxg^{-1}=y$ shows that $W\equiv (g_1,\dots,g_n,x_1,\dots,x_k,g_n^{-1},\dots,g_1^{-1})$,
representing the left-hand side, cannot be a reduced word. Since the word $(g_1,\dots,g_n)$ is reduced and
$\vls_\ga(g) \cap \supp_\ga(x)=\vfs_\ga(g^{-1}) \cap \supp_\ga(x)=\emptyset$, there must exist $m,m' \in \{1,\dots,n\}$ such that  $g_m$ can be joined with
$g_{m'}^{-1}$ in $W$. It follows that $m=m'$ because otherwise the word $(g_1,\dots,g_n)$ would not be reduced. It also follows that $g$ is represented by the word
$(g_1,\dots,g_{m-1},g_{m+1},\dots,g_n,g_m)$ and $g_m \in G_{\vls_\ga(g)} \cap G_S$, where $S=\link_\ga(\supp_\ga(x))$. Hence we have $g' x {g'}^{-1}=y$, where $g'\in G $ is the element represented by the word
$U \equiv (g_1,\dots,g_{m-1},g_{m+1},\dots,g_n)$. Note that the word $U$ is reduced as it was obtained from a reduced word for $g$ by removing its last syllable; in particular, $|g'|_\ga=n-1$.

Suppose that there is some  $v \in \vls_\ga (g') \cap \supp_\ga(x)$.  Then $g'=g''g_p$ for some $p \in \{1,\dots,n\}$, $p \neq m$, with $g_p \in G_v$ and $|g''|_\ga=n-2$. Observe that $g_m$
must commute with $g_p$, as $g_m \in G_{\link_\ga(v)}$, therefore $g=g'g_m=g''g_pg_m=g''g_mg_p$, implying that $v \in \vls_\ga(g)$, which contradicts to our assumptions. Thus $\vls_\ga (g') \cap \supp_\ga(x)=\emptyset$
and we can apply the induction hypothesis to conclude that $g'$ is represented by a reduced word $(h_1,\dots, h_r,h_{r+1}, \dots,h_{n-1})$, where $r \in \{1,\dots,n-1\}$,
$h_1,\dots,h_r \in G_T$ and $h_{r+1},\dots,h_{n-1} \in G_S$. After setting $h_n\coloneqq g_m$, we can conclude that $g=g'g_m$ is represented by the word $(h_1,\dots,h_{n-1},h_n)$
(which must be reduced as its length is $n=|g|_\ga$), satisfying the required properties.
\end{proof}

\begin{lemma} \label{lem:inter_spec} Let $\mathcal S$ be an arbitrary collection of subsets of $V=V\ga$.
Then $\bigcap_{S\in\mathcal{S}}G_S=G_T$, where $T\coloneqq \bigcap_{S\in\mathcal{S}} S \subseteq V$.
\end{lemma}

\begin{proof} Evidently, $G_T \leqslant \bigcap_{S\in\mathcal{S}}G_S$. To show the reverse inclusion, consider any $x \in \bigcap_{S\in\mathcal{S}}G_S$. By Theorem~\ref{thm:norm_form},
$\supp_\ga(x) \subseteq S$ for every $S\in \mathcal{S}$, hence $\supp_\ga(x) \subseteq\bigcap_{S\in\mathcal{S}} S=T$, thus $x \in G_T$.
\end{proof}

\begin{prop}\label{prop:inter_parab} 
Consider arbitrary $S,T \subseteq V$ and $g \in G$. Then there exist $P \subseteq S \cap T$ and $h \in G_T$ such that $gG_Sg^{-1} \cap G_T = hG_Ph^{-1}$.
\end{prop}

\begin{proof} Obviously one can write $g=hg'$, where $|g|_\ga=|h|_\ga+|g'|_\ga$, $\supp_\ga(h) \subseteq T$ and $\vfs(g') \cap T =\emptyset$. Then
$h \in G_T$ and $gG_Sg^{-1} \cap G_T=h\left(g'G_S{g'}^{-1} \cap G_T \right)h^{-1}$.
If $g'=g''h'$ for some $h' \in G_S$ with $|g'|_\ga=|g''|_\ga+|h'|_\ga$, then evidently $g'G_S{g'}^{-1}=g'' G_S{g''}^{-1}$, that is $g'$ can be shortened. Thus without loss of
generality we can assume that $\vls(g') \cap S=\emptyset$.

For every element $y \in g'G_S{g'}^{-1} \cap G_T$, take some $x \in G_S$ such that
$g'x{g'}^{-1}=y$. Applying Lemma~\ref{lem:sm_aux} we see that $\supp_\ga(g') \subseteq \link_\ga(\supp_\ga(x))$. The latter implies that $\supp_\ga(x) \subseteq \link_\ga(\supp_\ga(g'))$,
i.e., every syllable of $x$ commutes with every syllable of $g'$. Setting $Q\coloneqq S \cap \link_\ga(\supp_\ga(g')) \subseteq V$ we see that $$g'G_S{g'}^{-1} \cap G_T = g'G_Q{g'}^{-1} \cap G_T=G_Q \cap G_T.$$

By Lemma~\ref{lem:inter_spec}, $G_Q \cap G_T=G_P$, where $P\coloneqq Q \cap T$, and so $$gG_Sg^{-1} \cap G_T =h \left(g'G_S{g'}^{-1} \cap G_T\right) h^{-1}= hG_Ph^{-1},$$
as claimed.
\end{proof}

\begin{df} Any subgroup  $K \leqslant G=\GG$, conjugate to a full subgroup, is called \textit{parabolic}.
Moreover, if there are $S \subsetneqq V$ and $f \in G$ such that  $K=fG_S f^{-1}$, then $K$ is said to be a \textit{proper parabolic subgroup}.
In the latter case $K \neq G$ unless $G_v=\{1\}$ for all $v \in V-S$.
Clearly any parabolic subgroup $K=fG_S f^{-1}$, where $S \subseteq V$ and $f \in G$, is a retract of $G$, with the retraction $\rho_K\colon G \to K$, defined by
$\rho_K(g)\coloneqq  f\rho_S(f^{-1}gf)f^{-1}$ for all $g \in G$.
\end{df}

Proposition~\ref{prop:inter_parab} immediately yields
\begin{cor}\label{cor:inter_parab} The intersection of two parabolic subgroups of a graph product is again a parabolic subgroup.
\end{cor}

\begin{lemma} \label{lem:conj_of_parab_in_another_parab} If $S,T \subseteq V$ and $g_1,g_2 \in G$ satisfy $g_2G_Tg_2^{-1} \leqslant g_1G_Sg_1^{-1}$ then $G_T=G_{T\cap S}$. If, in addition,
$g_1=1$ then there is $g_3 \in  G_S$ such that $g_2G_Tg_2^{-1}= g_3G_{T\cap S}g_3^{-1}$.
\end{lemma}

\begin{proof} Denote $K\coloneqq g_1G_Sg_1^{-1}$ and let $\rho_K\colon G \to K$ be the retraction defined above. By definition, $G_{V-S} \subseteq \ker\rho_S=\ker\rho_K$, and as  $g_2G_{T-S}g_2^{-1} \leqslant K$, we have
$g_2G_{T-S}g_2^{-1}=\rho_K(g_2G_{T-S}g_2^{-1})=\{1\}$. Therefore $G_{T-S}=\{1\}$, and $G_T=\gen{G_{T\cap S},G_{T-S}}=G_{T\cap S}$. It follows that $G_T \leqslant G_S$ and $\rho_S(G_T)=G_T=G_{T\cap S}$.

Now, if $g_1=1$, we have $g_2G_{T}g_2^{-1} \leqslant K=G_S$. Hence $g_2G_{T}g_2^{-1} =\rho_S(g_2G_{T}g_2^{-1})=g_3G_{T\cap S}g_3^{-1}$,
where $g_3\coloneqq \rho_S(g_2) \in G_S$.
\end{proof}

\begin{cor} \label{cor:non-tr_parab_in_other} If $S,T \subseteq V$ and $g_1,g_2 \in G$ are such that $g_2G_Tg_2^{-1} \leqslant g_1G_Sg_1^{-1}$ and $G_t \neq \{1\}$ for all $t\in T$ then
$T \subseteq S$.
\end{cor}

\begin{lemma} \label{lem:conj_of_parab} If $K$ is a parabolic subgroup of a graph product $G$ and $gKg^{-1} \subseteq K$ for some $g \in G$, then $gKg^{-1}=K$.
\end{lemma}

\begin{proof} Let $\rho_K\colon G \to K$ be a retraction of $G$ onto $K$. For any $h \in K$, take $h'\coloneqq g\rho_K(g^{-1}) h \rho_K(g) g^{-1} \in gKg^{-1} \subseteq K$. Then
$h'=\rho_K(h')=h$, hence $h \in gKg^{-1}$ for all $h \in K$, thus $K=gKg^{-1}$.
\end{proof}

\begin{prop}\label{prop:min_parab} Let $X$ be a subset of the graph product $G=\GG$ such that at least one of the following conditions holds:
\begin{itemize}
	\item[(i)] the graph $\Gamma$ is finite;
	\item[(ii)] the subgroup $\gen{X} \leqslant G$ is finitely generated.
\end{itemize}
Then there exists a unique minimal parabolic subgroup of $G$ containing $X$.
\end{prop}

\begin{proof} The uniqueness is clear from Corollary~\ref{cor:inter_parab}.

To establish the existence, suppose, at first, that $\Gamma$ is finite. Then there is a subset $S \subseteq V$, of minimal cardinality, such that $X \subseteq g_1G_Sg_1^{-1}$ for some $g_1 \in G$.
Let us show that $g_1G_Sg_1^{-1}$ is a minimal parabolic subgroup containing $X$. Suppose that there is some parabolic subgroup $g_2G_Tg_2^{-1}$ of $G$ such that
$X \subseteq g_2G_Tg_2^{-1} \subseteq g_1G_Sg_1^{-1}$. According to Lemma~\ref{lem:conj_of_parab_in_another_parab}, $G_T=G_{T\cap S}$, and the minimality of $|S|$ implies that
$|S|\le |T\cap S|<\infty$. Consequently, $S=T\cap S$ and $G_S=G_{T\cap S}=G_T$. Therefore, applying Lemma~\ref{lem:conj_of_parab}, we can conclude that $g_2G_Tg_2^{-1} = g_2G_Sg_2^{-1}=g_1G_Sg_1^{-1}$, as required.

In the case (ii), when $\gen{X}=\gen{X'}$ for some finite subset $X' \subset G$, it is clear that $X \subseteq \gen{X'} \leqslant G_{A}$, where $A\coloneqq \supp_\ga(X')$ is a finite subset of $V$.
Recall that the full subgroup $G_A$, of $G$, is itself a graph product (of the groups $\{G_v \mid v \in A\}$ with respect to the graph $\ga_A$) and any parabolic subgroup of
$G_A$ is a parabolic subgroup of $G$. Since $|A|<\infty$, by the first part of the proof $X$ is contained in a minimal parabolic subgroup $g_1G_Sg_1^{-1}$ of $G_A$
for some $S \subseteq A$ and $g_1 \in G_A$. If $X \subseteq g_2G_T g_2^{-1} \subseteq g_1G_Sg_1^{-1}$ for some $T \subseteq V$ and $g_2 \in G$, then $g_2G_Tg_2^{-1} \subseteq G_A$ and so, by
Lemma~\ref{lem:conj_of_parab_in_another_parab}, $g_2G_Tg_2^{-1}$
is a parabolic subgroup of $G_A$. Therefore, using minimality of $g_1G_Sg_1^{-1}$, we achieve $g_2G_T g_2^{-1} = g_1G_Sg_1^{-1}$; thus the proposition is proved.
\end{proof}

\begin{df}  Suppose that a subset $X \subseteq G=\GG$ is contained in a minimal parabolic subgroup of $G$.
Then this parabolic subgroup will be called the \textit{parabolic closure} of $X$ and will be denoted by $\pc_\ga(X)$.
\end{df}

Recall that the normalizer $\No_G(X)$, of a subset $X \subseteq G$, is the subgroup of $G$ defined  by $\No_G(X)\coloneqq \{g \in G \mid gXg^{-1}=X\}$.

\begin{lemma} \label{lem:parab_normal} Let $G=\GG$ and $X \subseteq G$. Suppose that the parabolic closure of $X$ in $G$ exists.
Then for any $g \in G$ with $X \subseteq gXg^{-1}$ one has $gLg^{-1}=L$, where $L\coloneqq \pc_\ga(X)$. In particular,
$\No_G(X) \leqslant \No_G(L)$, i.e., any element normalizing $X$ also normalizes $L$.
\end{lemma}

\begin{proof} 
Clearly for any $g \in G$, with $X \subseteq gXg^{-1}$, $X$ is contained in $gLg^{-1}$, which is also a parabolic subgroup. Hence
$L \subseteq gLg^{-1}$, and Lemma~\ref{lem:conj_of_parab} implies that $L=gLg^{-1}$, thus $g \in \No_G(L)$.
\end{proof}

\begin{prop}\label{prop:normalizer_of_parab} Let $K$ be a non-trivial parabolic subgroup of the graph product $G=\GG$. Choose $f \in G$ and $S\subseteq V$
so that $K=fG_Sf^{-1}$ and $G_{s} \neq \{1\}$ for all $s \in S$. Then $\No_G(K)=f G_{S \cup \link_\ga(S)}f^{-1}$; in particular the normalizer $\No_G(K)$ is a parabolic subgroup of $G$.
\end{prop}

\begin{proof} After conjugating everything by $f^{-1}$, we can assume that $K=G_S$. Observe that $S \neq \emptyset$ since $K \neq \{1\}$.
Consider any element $g \in \No_G(G_S)$, any $s \in S$ and any $x \in G_{s}\setminus\{1\}$. Then $gxg^{-1}=y$ for some $y \in G_S$ and Lemma~\ref{lem:sm_aux} implies that $g \in G_{S \cup \link_\ga(s)}$.
Since the latter holds for every $s \in S$, we see that $g \in \bigcap_{s \in S} G_{S \cup \link_\ga(s)}=G_T$, where $T=\bigcap_{s \in S} (S \cup \link_\ga(s))=S \cup \link_\ga(S)$ by Lemma~\ref{lem:inter_spec}.
Thus $\No_G(G_S) \leqslant G_{S \cup \link_\ga(S)}$, and evidently $G_{S \cup \link_\ga(S)} \leqslant \No_G(G_S)$, hence  $\No_G(G_S) = G_{S \cup \link_\ga(S)}$.
\end{proof}

For any graph $\ga$ one can define its \textit{complement graph} $\ga'$ to be the graph with the same vertex set $V\ga'=V\ga$ such that the edge set $E\ga'$ is the complement of $E\ga$ in the
set of two-element subsets of $V\ga$; in other words, for any $u,v \in V\ga'=V\ga$, $\{u,v\} \in E\ga'$ if and only if $\{u,v\} \notin E\ga$.

\begin{df}
A graph $\ga$ will be called \textit{irreducible} if $V\ga$ cannot be represented as a union of two disjoint non-empty subsets $A,B \subset V\ga$ such that $B =\link_\ga(A)$.
\end{df}

It is easy to see that $\ga$ is irreducible if and only if $\ga'$ is connected.
In the case when $\ga$ is not irreducible, any graph product $G=\GG$, with respect to $\ga$, naturally splits as a direct product
$G_A\times G_B$, where $V=A\sqcup B$ and $B=\link_\ga(A)$.

\begin{cor} \label{cor:norm_in_irred} Assume that $\ga$ is a finite irreducible graph and $G=\GG$ is the graph product of a family of groups $\G$ with respect to $\ga$.
Suppose that $N \leqslant G$ is a non-trivial subgroup such that $\pc_\ga(N)$ is a proper parabolic subgroup of $G$. Then $\No_G(N)$ is contained in a proper parabolic subgroup of $G$.
\end{cor}

\begin{proof} Combining Proposition~\ref{prop:min_parab} with Lemma~\ref{lem:parab_normal} we see that $\No_G(N) \leqslant \No_G(K)$, where $K\coloneqq \pc_\ga(N)$.
By the assumptions, $K=fG_Sf^{-1}$ for some $S \subsetneqq V$ and $f \in G$; after discarding all $s \in S$ with trivial vertex groups, we can suppose that $G_{s} \neq \{1\}$ for each $s \in S$.
Then $\No_G(K)=fG_{S\cup \link_\ga(S)}f^{-1}$, according to Proposition~\ref{prop:normalizer_of_parab}.
Since $S \neq\emptyset$ (as  $\{1\} \neq N \leqslant K$), we can conclude that $\No_G(K)$
is a proper parabolic subgroup of $G$ because $S\cup \link_\ga(S) \neq V$ as $\ga$ is irreducible.
\end{proof}

\begin{lemma}\label{lem:virt_cyc-parab} If $\ga$ is a finite irreducible graph and $H \leqslant G$ is a virtually cyclic subgroup which is not contained in any proper parabolic subgroup of $G$,
then $H \cap K$ is finite for each proper parabolic subgroup $K$ of $G$.
\end{lemma}

\begin{proof} Suppose that $H\cap K$ is infinite for some proper parabolic subgroup $K\leqslant G$. Then there is an infinite cyclic subgroup $N\leqslant H \cap K$
such that $N \lhd H$. Observe that $\pc_\ga(N) \leqslant K$ is a proper parabolic subgroup and $H \leqslant \No_G(N)$, therefore, $H$ must also be contained in a
proper parabolic subgroup by Corollary \ref{cor:norm_in_irred}, contradicting to our assumptions. Hence $|H\cap K|<\infty$.
\end{proof}

\begin{prop} \label{prop:fin_parab_cl} Suppose that the graph $\ga$ is finite and $X \subseteq G=\GG$. Then there is a finite subset $X' \subseteq X$ such that $\pc_\ga(X)=\pc_\ga(X')$.
\end{prop}

\begin{proof} By Proposition~\ref{prop:min_parab}, for every finite subset $Y \subseteq X$, there exist $S(Y) \subseteq V$ and $f=f(Y) \in G$ such that
$\pc_\ga(Y)=fG_{S(Y)}f^{-1}$ and $G_s \neq \{1\}$ for all $s \in S(Y)$.
Since $V=V\ga$ is finite, the function, that corresponds to each finite subset $Y \subseteq X$ the integer value $|S(Y)|$, attains its maximum
on some finite subset $X' \subseteq X$. Evidently, to prove the claim it is enough to show that $X \subseteq \pc_\ga(X')$.

For any $x \in X$ we can find $T \subseteq V$ and $g \in G$ such that $\pc_\ga(X' \cup \{x\})=gG_Tg^{-1}$ and $G_t \neq \{1\}$ for all $t \in T$.
By the choice of $X'$, we have $|T|\le|S|$ and $fG_Sf^{-1} \leqslant gG_Tg^{-1}$, where
$S:=S(X') \subseteq V$ and $f:=f(X') \in G$. It follows that 
$S \subseteq T$ (Corollary \ref{cor:non-tr_parab_in_other}), which, in view of $|T|\le|S|$, implies that $S=T$. Finally, the inclusion $fG_Tf^{-1} \leqslant gG_Tg^{-1}$ together with Lemma~\ref{lem:conj_of_parab}
yield $fG_Sf^{-1} =fG_Tf^{-1}= gG_Tg^{-1}$, allowing to conclude that $x \in\pc_\ga(X' \cup \{x\})= \pc_\ga(X')$ for all $x \in X$, thus finishing the proof.
\end{proof}

\begin{cor} If $\ga$ is finite and $G=\GG$ is the graph product of groups with respect $\ga$, then any descending/ascending chain of parabolic subgroups stabilizes.
\end{cor}

\begin{proof} This is an easy consequence of Propositions~\ref{prop:min_parab} and~\ref{prop:fin_parab_cl}, and is left as an exercise for the reader.
\end{proof}


\section{Subgroups that contain no {\nafs s}}\label{S:structuretheorem}
Throughout this section $\Gamma$ will be a finite simplicial graph, $V=V\ga$, $G=\GG$ will be the graph product of a family of groups $\G=\{G_v \mid v \in V\}$ with respect to $\ga$.

Suppose that $G$ acts on a simplicial tree $\mathcal T$ without edge inversions. It is well known that any element $g \in G$ either fixes a vertex of $\mathcal T$
or there exists a unique minimal $\gen{g}$-subtree of $\mathcal T$, which is a bi-infinite geodesic path (called \textit{the axis} of $g$), where $\gen{g}$ acts by translation
(see \cite[I.4.11,I.4.13]{Dicks-Dunwoody} or  \cite[1.3]{CullerMorgan}).
In the former case $g$ is said to be \textit{elliptic}, and in the latter case it is said to be \textit{hyperbolic}. The \textit{translation length} $\|g\|$, of $g$, is the integer defined by
$\|g\|\coloneqq \min\{d_{\mathcal T}(p,g \circ p) \mid p \mbox{ is a vertex of }\mathcal T\}$, where $d_{\mathcal T}$ is the standard edge-path metric on $\mathcal T$. If the element $g \in G$
is hyperbolic then $d_\mathcal{T}(p,g \circ p)=\|g\|>0$ for every vertex $p \in \axis(g)$ (see \cite[1.1.3]{CullerMorgan}).

The following theorem allows us to describe the structure of any subgroup of a graph product (with respect to a finite graph), which has no {\nafs s}.

\begin{thm}\label{thm:structure} Suppose that the graph $\ga$ is finite, irreducible and has at least two vertices.
If $H \leqslant G$ is a subgroup then at least one of the following is true:

\begin{enumerate}
	\item $H$ is contained in a proper parabolic subgroup of $G$;
	\item $H \cong \Z$;
	\item $H \cong \D_\infty$;
	\item \label{it:4} for every proper subset $S \subsetneqq V$ there is a subgroup $F=F(S) \leqslant H$ such that $F \cong\mathbb{F}_2$ and $F \cap gG_Sg^{-1}=\{1\}$ in $G$ for all $g \in G$.	
\end{enumerate}
\end{thm}

\begin{proof} Suppose that \eqref{it:4} does not hold, i.e., there is a subset $S\subsetneqq V$ such that
every \nafs{} of $H$ intersects non-trivially some conjugate of $G_S$ in $G$ (note that this is also the case if $H$ contains no {\nafs s} at all).
Choose any vertex $v \in V\setminus S$ and set $C\coloneqq \link_\ga(v)$, $B\coloneqq C \sqcup \{v\}$ and $A\coloneqq V-\{v\}$;
observe that $S \subseteq A$.
As $\ga$ is irreducible and $|V|\ge 2$, $A$ and $B$ are non-empty proper subsets of $V$.
Then $G$ naturally splits as the amalgamated free product: $G=G_A*_{G_C}G_B$. Let $\mathcal T$ be the Bass-Serre tree, corresponding to this splitting, on which $G$ acts simplicially and
without edge inversions. 
The vertex stabilizers of this splitting are conjugates of $G_A$ and $G_B$, and the edge stabilizers are conjugates of $G_C$ (see \cite[I.4.1, Thm. 7]{Serre}), all of which are proper
parabolic subgroups in $G$.

Since every \nafs{} of $H$ has non-trivial intersection with some conjugate of $G_A$ (and thus it cannot act freely on $\mathcal T$), by Theorem 2.7 from \cite{CullerMorgan}
the induced action of $H$ on $\mathcal T$ must be \textit{reducible}, which means that one of the following occurs:
\begin{itemize}
	\item[(i)] every element of $H$ fixes a vertex of $\mathcal T$;
	\item[(ii)] there is a bi-infinite geodesic line $\mathcal L$ which is invariant under the action of $H$;
	\item[(iii)] there is an end of $\mathcal T$ which is fixed by $H$.
\end{itemize}

In case (i), by Proposition~\ref{prop:fin_parab_cl}, there is a finite subset $X' \subseteq H$ such that $H \leqslant \pc_\ga(X')$.
The subgroup $H'\coloneqq \gen{X'} \leqslant H$ is finitely generated and every element of $H'$ fixes a vertex of $\mathcal T$, therefore
$H'$ fixes some vertex of $\mathcal T$ (see \cite[I.6.5, Cor. 3]{Serre}). But full $G$-stabilizers of vertices of $\mathcal T$ are proper parabolic subgroups of $G$, therefore
$\pc_\ga(X')=\pc_\ga(H')$ is a proper parabolic subgroup of $G$. Thus $H$ will be contained in a proper parabolic subgroup of $G$.

Assume, now, that we are in case (ii). Then $H$ acts on the bi-infinite geodesic line $\mathcal L$, giving rise to
a homomorphism $\phi: H \to \mathbb{D}_\infty$, where $\mathbb{D}_\infty$ is the infinite dihedral group of all simplicial isometries of $\mathcal L$.
Moreover, we can suppose that $\phi(H)$ is infinite, because otherwise $H$ would fix a vertex of $\mathcal L$ and we would be in case (i).
Denote $N\coloneqq \ker\phi \lhd H$; then $N$ fixes every vertex of $\mathcal L$ and so it is contained in a proper parabolic subgroup of $G$ by case (i).
If $N=\{1\}$ then $H$ is isomorphic to an infinite subgroup of $\D_\infty$, which is either infinite cyclic or infinite dihedral.
If $N \neq \{1\}$ then $H \leqslant \No_G(N)$ is contained in a proper parabolic subgroup of $G$ by Corollary~\ref{cor:norm_in_irred}.

Finally, assume that we are in case (iii), i.e., $H$ fixes some end $e$ of $\mathcal T$. If an element $g \in H$ fixes some vertex $o$ of $\mathcal T$, then it will have to fix every vertex of the
unique infinite geodesic ray between $o$ and $e$. If $h \in H$ is another elliptic element, then it will fix (pointwise) another geodesic ray converging to $e$. But any two rays
converging to $e$ are eventually the same, in particular they will have a common vertex, which will then be fixed by both $g$ and $h$. Thus $gh \in H$ will also be elliptic.
It follows that the subset $N \subseteq H$, of all elliptic elements of $H$, is a normal subgroup of $H$. The argument for case (i) shows that $N$ is contained in a proper  parabolic subgroup of $G$.
And if $N \neq \{1\}$ then  Corollary~\ref{cor:norm_in_irred} yields that $H \leqslant \No_G(N)$ is contained in a proper parabolic subgroup of $G$.

Thus, we can assume that $N=\{1\}$, in other words all non-trivial elements of $H$ are hyperbolic. Let $x \in H- \{1\}$ be an element of minimal translation length
(for the action of $H$ on $\mathcal T$). For any other element $y \in H-\{1\}$, the intersection of $\axis(x)$ and $\axis(y)$ is an infinite geodesic ray $\mathcal R$,
starting at some vertex $p$ of $\mathcal T$ and converging to $e$. After replacing $x$ and $y$ with their inverses, where necessary, we assume that
$x \circ \mathcal{R} \subset \mathcal{R}$ and $y \circ \mathcal{R} \subset \mathcal{R}$.
Write $\|y\|=m\|x\|+n$, where $m \in \N$, $n \in \N\cup\{0\}$ and $n <\|x\|$. Then $y \circ p, (x^{-m}y) \circ p \in \mathcal{R}$ and so
$d_{\mathcal T}\left(p,(x^{-m}y)\circ p\right)=\|y\|-m\|x\|=n$, thus $\|x^{-m}y\| \le n < \|x\|$, hence $x^{-m}y=1$
by the minimality of $\|x\|$. Thus $y \in \gen{x}$ for all $y \in H$, i.e., $H = \gen{x}$ is cyclic. Since $x$ is hyperbolic, it must have infinite order (cf. \cite[I.4.11]{Dicks-Dunwoody}),
implying that $H \cong \Z$.
\end{proof}

\begin{rem}\label{rem:4-core} In the assumptions of Theorem \ref{thm:structure}, assume that the condition \eqref{it:4} is satisfied. Then $H$ has a finitely generated
subgroup $M$ such that $M$ is not contained in any proper parabolic subgroup of $G$ and $M$ contains a free subgroup of rank $2$.
Indeed, simply let $M$ to be the subgroup of $H$ generated by $\{F(S) \mid S \subsetneqq V\}$.
\end{rem}

The above Structure Theorem can be used to give a more precise description of subgroups of graphs products that do not contain {\nafs s}.
Let $\mathcal{I}$ be a collection of cardinals. For $n=1,2,3,\dots$ we define classes of groups $\ISG_n$ as follows:

\noindent $\ISG_1$ is the class consisting of all $\mathcal{I}$-generated groups, which are isomorphic to subgroups of $G_v,$ $v\in V\ga,$ and which do not contain a copy of \ff.

\noindent $\ISG_2$ is the smallest class, closed under isomorphisms, containing all $\mathcal{I}$-generated groups $H$, such that $H$ is a subgroup of $A \times B$ for some
$A,B \in \ISG_1$, and satisfying the next conditions: $\Z \in\ISG_2$ and, if  for some $v\in V\ga$ the vertex group $G_v$ has $2$-torsion, then $\D_{\infty} \in \ISG_2$.

\noindent $\ISG_n$, $n\geq 3$, is the class consisting of all $\mathcal{I}$-generated groups $H$, such that $H$ is isomorphic to a subgroup of $A \times B$ for some
$A \in \ISG_{l}$ and $B \in \ISG_m$ with $l,m \in \N$, $l+m=n$.

\begin{thm}\label{thm:sharpA} Suppose that $\mathcal{I}$ is a collection of cardinals, $\ga$ is a non-empty finite graph and $G=\GG$ is
the graph product of a family of groups $\G$ with respect to $\ga$.
If $H \leqslant G$ is an $\mathcal{I}$-generated subgroup, which does not contain a copy of \ff, then $H$ is in the class $\ISG_{|V\ga|}.$
\end{thm}
\begin{proof} Observe that $\ISG_{n-1} \subseteq \ISG_{n}$ whenever $n \ge 2$.
The proof argues by induction on the number of vertices  in $\ga.$ The claim is evident if $|V\ga|\le 1$. So we will assume that $|V\ga|\ge 2$ and that the result holds for
every graph product over a graph with fewer vertices than $\ga.$

Let $H$ be an $\mathcal{I}$-generated subgroup of $G$ that does not contain a \nafs; we have to show that $H \in \ISG_{|V\ga|}.$
There are two cases to consider.

{\it Case 1: $\ga$ is irreducible.} Since $|V\ga|\ge 2,$  we are under the assumptions of the Structure Theorem (Theorem~\ref{thm:structure}).
Therefore $H$ is either contained in a proper parabolic subgroup  (and, hence, the result follows by induction), or
$H\cong \Z$ or $H \cong \D_\infty$. By the assumptions, $\Z\in \ISG_{n}$  for all $n\geq 2$. On the other hand, the case $H\cong \D_\infty$ can only happen if $G$ has $2$-torsion,
which, in its own turn, only happens when
$G_v$ has $2$-torsion for some $v \in V\ga$ (see, for example, Corollary~\ref{cor:torfree} below), and then $\D_\infty$ is in $\ISG_{n}$ for any $n\geq 2$.

{\it Case 2: $\ga$ is not irreducible.} Then there exist non-empty subsets $A$ and $B$ of $V\ga$ such that $G=G_A\times G_B$ and $|A|+|B|=|V\ga|$.
As $H$ is $\mathcal{I}$-generated, so are $\rho_A(H)$ and $\rho_B(H).$  Since $H$ does not contain a \nafs, neither $\rho_A(H)$ nor $\rho_B(H)$ contain a copy of \ff.
Using the induction hypothesis, $\rho_A(H)\in \ISG_{|A|}$ and  $\rho_B(H) \in \ISG_{|B|}$. It remains to observe that $H$ is contained in the subgroup of $G$,
generated by $\rho_A(H)$ and $\rho_B(H)$, which is naturally isomorphic to  $\rho_A(H)\times\rho_B(H)$, as required.
\end{proof}

Theorem A is almost an immediate consequence of Theorem \ref{thm:sharpA}.
\begin{proof}[Proof of Theorem A] If the collection $\mathcal I$ contains no non-zero cardinals, then the statement trivially holds (since only the trivial group could be
$\mathcal{I}$-generated in this case), therefore we will further assume that $\mathcal I$ contains at least one non-zero cardinal, implying that any cyclic group is
$\mathcal I$-generated.

By the assumptions, all vertex groups $G_v$, $v \in V\ga$, satisfy the Tits Alternative relative to  ($\mathcal{I},\mathcal{C}$). This means that  $\ISG_1 \subseteq \mathcal{C}$.
By (P$0$),(P$1$) and (P$2$), $\mathcal{I}$-generated subgroups of direct products of groups from $\ISG_1$ are in $\mathcal{C}$. By (P$3$), $\Z \in \mathcal{C}$;
and if $G_v$ has $2$-torsion for some $v\in V\ga,$ then $\Z/2\Z \in \mathcal{C}$ (according to the Tits Alternative for $G_v$) and so  $\D_{\infty} \in \mathcal{C}$ by (P$4$).
Therefore $\ISG_2 \subseteq \mathcal{C}.$ Again, by (P$0$),(P$1$) and (P$2$), for $n\geq 3,$ all groups from $\ISG_n$ are in $\mathcal{C}.$

Now, consider an arbitrary $\mathcal{I}$-generated subgroup $H \leqslant G=\GG$, which does not contain a \nafs.
Without loss of generality we can suppose that the graph $\ga$ is non-empty, and so, by Theorem~\ref{thm:sharpA}, $H \in \ISG_{|V\ga|} \subseteq \mathcal{C}$,
finishing the proof.
\end{proof}

Taking $\mathcal I$ to be the collection of all countable cardinals, one can use Theorem \ref{thm:sharpA} to obtain another interesting corollary,
the proof of which is left for the reader:

\begin{cor}\label{cor:rank_ab_in_RAAG}
If $\ga$ is a finite graph and $G$ is a right angled Artin group corresponding to $\ga$ (that is, $G$ is the graph product of infinite cyclic groups with respect to $\ga$)
then the rank  of any abelian subgroup of $G$ does not exceed $|V\ga|$.
\end{cor}


\section{Kernel of the canonical retraction onto a vertex group}\label{S:kernel}
Let $\ga$ be a simplicial graph and let $G=\graphx\mathfrak{G}$ be a graph products of a family of groups $\G=\{G_v \mid v \in V\graphx\}$, with respect to $\ga$.
In this section we will use the Normal Form Theorem (Theorem \ref{thm:norm_form}) to prove that for any $a \in V\graphx$, the kernel
$\ker \rho_{\{a\}}$ is itself a graph product.
Let $V=V\graphx$, $A=\{a\}$, $B=\link_\graphx(a)$ and $C=V-(A\cup B).$ Then $V$ is the disjoint union of $A,B,$ and $C.$

For every $g \in G_A$ take a copy $\graphx^g_{B\cup C}$ of $\graphx_{B\cup C}$, then construct the graph $\graphz$ by gluing together the graphs  $\{\graphx^g_{B\cup C}\}_{g \in G_A}$ along the
corresponding copies of $\graphx^g_B$. More precisely, the vertex set of $\graphz$ is the set of equivalence classes of pairs $\{[(g,u)] \mid g \in G_A \mbox{ and } u \in {B\cup C}\}$,
where two such pairs $(g,u)$ and $(h,v)$ are equivalent $(g,u)\sim (h,v)$, if and only if $u=v \in B$. Two vertices $s,t \in V\graphz$ are connected by an
edge if and only if there are $g \in G_A$ and $u,v \in B\cup C$ such that $s=[(g,u)]$, $t=[(g,v)]$ and $\{u,v\} \in E\graphx$. Note that if $G_A$ is infinite and $C \neq \emptyset$, then the new graph $\graphz$
will have infinitely many vertices.

\begin{Ex}\label{ex:1}
Let $\graphx$ be the simple path of length $2$.  That is, $\graphx$ has three vertices $\{v_1,v_2,v_3\}$ and two edges $\{\{v_1,v_2\},\{v_2,v_3\}\}.$
Suppose that $G_{v_1}= \mathbb{Z}/3\mathbb{Z}=\{\bar 0,\bar 1, \bar 2\}$.
For $a=v_1$ we have $A=\{v_1\}$, $B=\{v_2\}$ and $C=\{v_3\}.$ Figure~\ref{fig:1} below illustrates the construction of the graph $\graphz$ in this case.

\begin{figure}[ht]
\setlength{\unitlength}{4144sp}%
\begingroup\makeatletter\ifx\SetFigFont\undefined%
\gdef\SetFigFont#1#2#3#4#5{%
  \reset@font\fontsize{#1}{#2pt}%
  \fontfamily{#3}\fontseries{#4}\fontshape{#5}%
  \selectfont}%
\fi\endgroup%
\begin{picture}(5167,2023)(249,-1259)
\put(4276,-151){\makebox(0,0)[b]{\smash{{\SetFigFont{10}{12.0}{\rmdefault}{\mddefault}{\updefault}{\color[rgb]{0,0,0}$[(\overline{0},v_2)]$}%
}}}}
{\color[rgb]{0,0,0}\thicklines
\put(1027, 73){\circle*{90}}
}%
{\color[rgb]{0,0,0}\put(1748, 73){\circle*{90}}
}%
{\color[rgb]{0,0,0}\put(2739,-557){\circle*{90}}
}%
{\color[rgb]{0,0,0}\put(2739, 73){\circle*{90}}
}%
{\color[rgb]{0,0,0}\put(2737,703){\circle*{90}}
}%
{\color[rgb]{0,0,0}\put(3459, 73){\circle*{90}}
}%
{\color[rgb]{0,0,0}\put(3459,703){\circle*{90}}
}%
{\color[rgb]{0,0,0}\put(3459,-557){\circle*{90}}
}%
{\color[rgb]{0,0,0}\put(4452, 74){\circle*{90}}
}%
{\color[rgb]{0,0,0}\put(5174, 74){\circle*{90}}
}%
{\color[rgb]{0,0,0}\put(5174,704){\circle*{90}}
}%
{\color[rgb]{0,0,0}\put(5176,-556){\circle*{90}}
}%
{\color[rgb]{0,0,0}\put(3466,704){\line(-1, 0){765}}
}%
{\color[rgb]{0,0,0}\put(3466, 74){\line(-1, 0){765}}
}%
{\color[rgb]{0,0,0}\put(3466,-556){\line(-1, 0){765}}
}%
{\color[rgb]{0,0,0}\put(5221, 74){\line(-1, 0){765}}
}%
{\color[rgb]{0,0,0}\put(1756, 74){\line(-1, 0){765}}
}%
{\color[rgb]{0,0,0}\put(1036, 74){\line(-1, 0){765}}
}%
{\color[rgb]{0,0,0}\put(4456, 74){\line( 6, 5){734.754}}
}%
{\color[rgb]{0,0,0}\put(4436, 63){\line( 6,-5){734.754}}
}%
\put(316,-151){\makebox(0,0)[b]{\smash{{\SetFigFont{10}{12.0}{\rmdefault}{\mddefault}{\updefault}{\color[rgb]{0,0,0}$v_1$}%
}}}}
\put(1036,-151){\makebox(0,0)[b]{\smash{{\SetFigFont{10}{12.0}{\rmdefault}{\mddefault}{\updefault}{\color[rgb]{0,0,0}$v_2$}%
}}}}
\put(1756,-151){\makebox(0,0)[b]{\smash{{\SetFigFont{10}{12.0}{\rmdefault}{\mddefault}{\updefault}{\color[rgb]{0,0,0}$v_3$}%
}}}}
\put(2746,479){\makebox(0,0)[b]{\smash{{\SetFigFont{10}{12.0}{\rmdefault}{\mddefault}{\updefault}{\color[rgb]{0,0,0}$(\overline{2},v_2)$}%
}}}}
\put(3466,479){\makebox(0,0)[b]{\smash{{\SetFigFont{10}{12.0}{\rmdefault}{\mddefault}{\updefault}{\color[rgb]{0,0,0}$(\overline{2},v_3)$}%
}}}}
\put(3466,-781){\makebox(0,0)[b]{\smash{{\SetFigFont{10}{12.0}{\rmdefault}{\mddefault}{\updefault}{\color[rgb]{0,0,0}$(\overline{0},v_3)$}%
}}}}
\put(2746,-151){\makebox(0,0)[b]{\smash{{\SetFigFont{10}{12.0}{\rmdefault}{\mddefault}{\updefault}{\color[rgb]{0,0,0}$(\overline{1},v_2)$}%
}}}}
\put(2746,-781){\makebox(0,0)[b]{\smash{{\SetFigFont{10}{12.0}{\rmdefault}{\mddefault}{\updefault}{\color[rgb]{0,0,0}$(\overline{0},v_2)$}%
}}}}
\put(3466,-151){\makebox(0,0)[b]{\smash{{\SetFigFont{10}{12.0}{\rmdefault}{\mddefault}{\updefault}{\color[rgb]{0,0,0}$(\overline{1},v_3)$}%
}}}}
\put(3061,-1186){\makebox(0,0)[b]{\smash{{\SetFigFont{12}{14.4}{\rmdefault}{\mddefault}{\updefault}{\color[rgb]{0,0,0}(b) $\Gamma_{B\cup C}^{\overline{0}}\cup \Gamma_{B\cup C}^{\overline{1}} \cup \Gamma_{B\cup C}^{\overline{2}} $}%
}}}}
\put(991,-1186){\makebox(0,0)[b]{\smash{{\SetFigFont{12}{14.4}{\rmdefault}{\mddefault}{\updefault}{\color[rgb]{0,0,0}(a) $\Gamma$}%
}}}}
\put(4861,-1186){\makebox(0,0)[b]{\smash{{\SetFigFont{12}{14.4}{\rmdefault}{\mddefault}{\updefault}{\color[rgb]{0,0,0}(c) $\Delta$}%
}}}}
\put(5401,479){\makebox(0,0)[b]{\smash{{\SetFigFont{10}{12.0}{\rmdefault}{\mddefault}{\updefault}{\color[rgb]{0,0,0}$[(\overline{2},v_3)]$}%
}}}}
\put(5401,-151){\makebox(0,0)[b]{\smash{{\SetFigFont{10}{12.0}{\rmdefault}{\mddefault}{\updefault}{\color[rgb]{0,0,0}$[(\overline{1},v_3)]$}%
}}}}
\put(5401,-781){\makebox(0,0)[b]{\smash{{\SetFigFont{10}{12.0}{\rmdefault}{\mddefault}{\updefault}{\color[rgb]{0,0,0}$[(\overline{0},v_3)]$}%
}}}}
{\color[rgb]{0,0,0}\put(315, 74){\circle*{90}}
}%
\end{picture}%
\caption{\label{fig:1} Illustration of Example \ref{ex:1}. }
\label{fig:XYZ}
\end{figure}
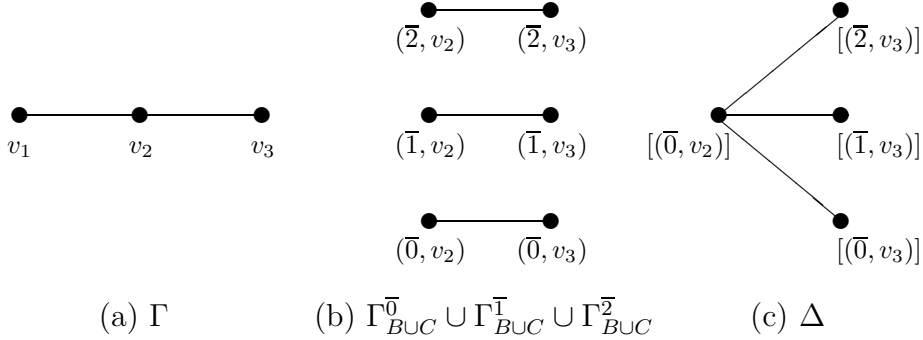

\end{Ex}

For each $[(g,v)]\in V\graphz,$ let $K_{[(g,v)]}$ be an isomorphic copy of $G_v$ (with a fixed isomorphism), and  for any $k\in K_{[(g,v)]},$ let $\overline{k}$  denote the corresponding element in $G_v$.
Set $\mathfrak{K}=\{K_t \mid t \in V\graphz \}$ and let $\graphz\mathfrak{K}$ be the corresponding graph product.
Define a map $\phi_0\colon \bigcup_{t \in V\graphz}{K_t} \to \graphx\mathfrak{G}$ by $\phi_0(k)=g\overline{k}g^{-1}$ whenever $k\in K_{[(g,v)]}.$ Note that this is well defined, since if $(g,v)\sim (h,u)$
then $u=v\in B$ and $g\overline{k}g^{-1}=\overline{k}=h\overline{k}h^{-1}.$ Clearly $\phi_0$ can be uniquely extended to a homomorphism
$\phi\colon \graphz \mathfrak{K} \to \graphx \mathfrak{G}$, because any defining relation of $\graphz \mathfrak{K}$ is mapped to a relation of $\graphx \mathfrak{G}$.

\begin{thm}\label{thm:kernel}
Suppose that $G\coloneqq \graphx\mathfrak{G}$ is a graph product of groups as above and $a\in V\graphx$ is a vertex of $\graphx$.
Then there is a group isomorphism $\psi\colon \ker \rho_{\{a\}} \to \graphz \mathfrak{K}$, where the graph $\graphz$ and the family of groups $\mathfrak{K}$ are defined above.
In particular, every vertex group $K_t$, $t \in V\graphz$, is isomorphic to some $G_v$, $v \in V\graphx$.

Moreover, for every $g \in \ker \rho_{\{a\}}$ we have
\begin{equation}\label{eq:comp} \text{ $|\psi(g)|_\graphz \leq|g|_\graphx$, and if $a \in \supp_\graphx(g)$
 then $|\psi(g)|_\graphz \leq|g|_\graphx-2$. }
\end{equation}
\end{thm}

\begin{proof}
Let $\phi\colon \graphz \mathfrak{K} \to \graphx \mathfrak{G}$ be the homomorphism defined above. We will show that $\phi$ is injective and its image is $\ker\rho_{\{a\}}$. We will then
take $\psi$ to be the inverse of $\phi$.

\noindent\textbf{Injectivity of $\phi$}

Let $W\equiv(k_1,\dots, k_n)$  be a reduced word representing a non trivial element of $\graphz \mathfrak{K}$.
For any given $i \in \{1,\dots,n\}$, we know that $\phi(k_i)=g_i \overline{k}_i g_i^{-1}$ for some $g_i \in G_A$. Suppose that $k_i \in K_{[(f_i,v_i)]}$ for some $v_i \in B \cup C$. If $v_i \in C$
then $g_i=f_i$ is uniquely determined by the equivalence class $[(f_i,v_i)]$. Otherwise, if $v_i \in B$ (this is equivalent to $\overline{k}_i \in G_{\link_\graphx(a)}$),
we have freedom to (recursively) choose $g_i \in G_A$ as follows: if $i=1$ we take $g_1=1$, and if $i>1$ we take $g_i=g_{i-1}$.

Then $\phi$ maps the element of $\graphz \mathfrak{K}$, represented by $W$, to the element of $\graphx \mathfrak{G}$, represented by the word
$W_1\equiv(g_1,\overline{k}_1,g^{-1}_1g_2,\overline{k}_2,g^{-1}_2g_3,\dots, g_{n-1}^{-1}g_n,\overline{k}_n,g^{-1}_n)$.
Arguing by contradiction, suppose that $W_1=1$ in $\graphx \mathfrak{G}$.
Let $W_2$ be the word obtained from $W_1$ by removing all the trivial $G_A$-syllables.
Since $W_2=1$ and $n \ge 1$, it follows that two syllables of $W_2$ can be joined together. However, from the choice of $g_i$'s it is clear that no two (non-trivial) $G_A$-syllables of $W_2$
can be joined together (because if $\overline{k}_i \in G_B$ and $i>1$ then $g_{i-1}^{-1}g_i=1$).
Therefore there must exist integers $i,j$, $1 \le i<j \le n$, such that $\overline{k}_i$ can be joined with $\overline{k}_j$ in $W_2$. There are two cases to consider.

\textit{Case 1.} $\overline{k}_i,\overline{k}_j \in G_v$ for some $v \in B$. Since these two syllables can be joined together in $W_2$,
we have $\overline{k}_m \in G_{\link_\graphx (v)}$ for each $m \in \{i+1,\dots,j-1\}$. Hence $k_i,k_j \in K_t$, where $t=[(1,v)] \in V\graphz$, and for every
$m \in \{i+1,\dots,j-1\}$ one has $k_m \in K_{\link_\graphz (t)}$. In other words, the syllables $k_i$ and $k_j$ can be joined together in $W$, contradicting the assumption that $W$ is reduced.

\textit{Case 2.} $\overline{k}_i,\overline{k}_j \in G_v$ for some $v \in C$. Since $A \cap \link_\graphx(v) =\emptyset$ we can conclude that for all $l$, $i \le l \le j-1$, $g_l=g_{l+1}$ in $G_A$, and
$\overline{k}_m \in G_{\link_\graphx (v)}$ for every $m \in \{i+1,\dots,j-1\}$.
Therefore $k_i,k_j \in K_t$, where $t=[(g_i,v)] \in V\graphz$, and for every $m \in \{i+1,\dots,j-1\}$ one has $k_m \in K_{\link_\graphz (t)}$. This again yields a contradiction with the assumption that $W$ was reduced.

Thus $W_1\neq 1$ in $\graphx\mathfrak{G}$, implying that $\phi$ is injective.

\noindent\textbf{The image of $\phi$ is $\ker \rho_A$}

First notice that each generator of $ \graphz \mathfrak{K}$ is mapped under $\phi$ to an element of $\ker \rho_A$. Therefore $\text{Im} \phi \subseteq \ker \rho_A.$

On the other hand, any element $g\in \ker \rho_A$ can be represented by a word
$W \equiv (S_0,h_1,S_1,\dots, h_n,S_n)$, where $n \ge 0$, every $S_i$ is a subword without $G_A$-syllables and each $h_j \in G_A-\{1\}$.
As $g\in \ker \rho_A$ we see that $\rho_A(g)=h_1h_2\cdots h_{n}=1$ in $G_A$.
For an element $h \in G_A$ and a word $S \equiv (s_1, \dots,s_m)$, with $s_i \in G_{u_i}$ for some $u_i \in B\cup C$, $i=1,\dots,m$,
we define $S^h$ to be the word $(h,s_1,h^{-1},h,s_2,h^{-1},\dots,h,s_m,h^{-1})$.
Observe that $S^h$ represents the $\phi$-image of the word $(k_1,\dots, k_m)$, where $k_i \in K_{[(h,u_i)]}$ and $\overline{k}_i=s_i$ for all $i$,
and the length of this word is the same as the length of $S$.

Note that in $\graphx \mathfrak G$ the element $g$ is also represented by the word
$$
W'\equiv (S_0,S_1^{h_1},S_2^{h_1h_2},\dots, S_n^{h_1h_2 \cdots h_{n}}).
$$
Evidently $W'$ is the $\phi$-image of some word,
representing an element of $ \graphz \mathfrak{K}$. Hence $\phi$ is surjective.

\noindent \textbf{Rewriting in $ \graphz \mathfrak{K}$}

Let $\psi\colon \ker \rho_A \to \graphz \mathfrak{K}$ be the inverse of $\phi$. From the previous construction of $W'$ it is clear that
$|\psi(g)|_\graphz \le |g|_\graphx-n$. And if $a \in \supp(g)$, then $n \ge 2$ (as $h_1h_2\cdots h_n=1$ but $h_i \neq 1$ for $i=1,\dots,n$), proving the last claim of the theorem.
\end{proof}

\begin{df} Let $\ga$ be a simplicial graph  and let $G=\GG$ be the graph product of a family of groups $\G$ with respect to $\ga$. Given any subset $X \subseteq G$, such that
the parabolic closure of $X$ exists, the \textit{essential support} $\esupp_\ga(X)$, of $X$, is the minimal subset $S \subseteq V\ga$ such that $\pc_\ga(X)=gG_Sg^{-1}$ for some $g \in G$.
\end{df}

\begin{rem}\label{rem:es_supp} If $\pc_\ga(X)=gG_{T}g^{-1}$ for some $T \subseteq V\ga$ then $\pc_\ga(X)=gG_{T'}g^{-1}$, where $T':=T-\{t \in T \mid G_t=\{1\}\}$,
and, according to Corollary \ref{cor:non-tr_parab_in_other}, $\esupp_\ga(X)=T'$. In particular for every $s \in \esupp_\ga(X)$, $G_s \neq \{1\}$,
hence $\esupp_\ga(X) \subseteq \supp_\ga(X)$ by Corollary \ref{cor:non-tr_parab_in_other}.
\end{rem}

\begin{rem}\label{rem:es_supp-prop_parab} For any finite subset $X \subseteq G$, the essential supports $\esupp_\ga(X)$ and $\esupp_\ga(\gen{X})$
are well-defined by Proposition \ref{prop:min_parab},
and $\esupp_\ga(\gen{X})=\esupp_\ga(X)$ is a finite subset of $V\ga$.
Moreover, $\esupp_\ga(X)=V\ga$ if and only if  $X$ is not contained in a proper parabolic subgroup of $G$.
\end{rem}

\begin{lemma}\label{lem:proj_on_esupp} Let $H$ be a subgroup of $G=\GG$ such that $E:=\esupp_\ga(H)\subseteq V\ga$ is well-defined
(e.g., if $\ga$ is finite or if $H$ is finitely generated).
Then the canonical projection
$\rho_E:G \to G_E$ is injective on $H$ and $\esupp_{\ga_E}(\rho_E(H))=E$, i.e., $\rho_E(H)$ is not contained in a proper parabolic subgroup of $G_E$,
where $G_E$ is equipped with the natural graph product decomposition with respect to $\ga_E$.
\end{lemma}

\begin{proof} By the assumptions, $H \leqslant gG_Eg^{-1}$ for some $g \in G$ and $E$ is the smallest subset of $V\ga$ with this property.
Evidently the restriction of $\rho_E$ to $gG_Eg^{-1}$ acts as conjugation by $\rho_E(g)g^{-1}$. It follows that $\rho_E$ is injective on $H$
and $H$ is not contained in any proper parabolic subgroup of $G_E$.
\end{proof}

Given a graph $\graphz$ and a subset $S\subseteq V\graphz$, we will say that $S$ is \textit{irreducible} if the full subgraph $\graphz_S$ is irreducible. In other words, $S$
is not irreducible if and only if $S=P \sqcup Q$ such that $P,Q \neq\emptyset$ and $Q \subseteq \link_\graphz(P)$.

\begin{prop}\label{prop:irred_subset} Let $\ga$ be a finite irreducible graph with at least two vertices and let $G=\GG$ be the graph product of a family of groups $\G$ with respect to $\ga$.
Consider a finite subset $X \subseteq G$ such that $X$ is not contained in any proper parabolic subgroup of $G$ and $X \subseteq \ker \rho_{\{a\}}$ for some $a \in V\ga$.
Using Theorem \ref{thm:kernel}, the kernel $ \ker \rho_{\{a\}}$ can be identified with the graph product $K:= \graphz \mathfrak{K}$ as above.
Then the essential support $\esupp_\graphz(X)$ is irreducible and $|\esupp_\graphz(X)|\ge 2$.
\end{prop}

\begin{proof} We will use the same notation ($A$, $B$, $C$, etc.) as in the beginning of the section. Note that  
$ \esupp_\graphz(X)$ is well-defined by Proposition \ref{prop:min_parab}. Let us start with observing two facts.

Fact 1: for any given element $g \in G_A$, $\esupp_\graphz(X)$ cannot be contained in $\{[(g,u)] \mid u \in B\cup C\}$ (otherwise $X \leqslant gG_{B\cup C}g^{-1}$,
which is a proper parabolic subgroup of $G$).
Hence there are $g_1,g_2 \in G_A$ and $u_1,u_2 \in  C$ such that $g_1 \neq g_2$ and $[(g_1,u_1)],[(g_2,u_2)] \in \esupp_\graphz(X)$. It follows that
$[(g_1,u_1)] \neq [(g_2,u_2)]$, and, in particular, $|\esupp_\graphz(X)|\ge 2$.

Fact 2: for every $u \in B \cup C$ there is
$g \in G_A$ such that $[(g,u)] \in \esupp_\graphz(X)$.
Indeed, suppose that there is $u \in B \cup C$ such that  $X \subseteq fK_Sf^{-1}$ for some $f \in K$ and $S \subseteq V\graphz$ with $[(g,u)] \notin S$ for all $g \in G_A$.
It follows, by definition, that $K_S \subseteq \bigcup_{h \in G_A} h G_T h^{-1}$ where $T:=(B\cup C)-\{u\} \subseteq V\ga$.
Then $X \subseteq fG_{T\cup\{a\}}f^{-1}$, which is a proper parabolic subgroup of $G$, contradicting to our assumptions.

Arguing by contradiction, let us assume that $\esupp_\graphz(X)$ is not  irreducible, that is $\esupp_\graphz(X)=P \sqcup Q$ such that $P,Q \subseteq V\graphz$ are non-empty and
$Q\subseteq \link_\graphz(P)$ in $\graphz$.
If $[(f_1,v_1)],[(f_2,v_2)] \in \esupp_\graphz(X)$ are such that $f_1,f_2$ are distinct elements of $G_A$ and $v_1,v_2 \in C$, then these vertices are not adjacent
in $\graphz$, and thus they must both lie either in $P$ or in $Q$.
Without loss of generality we can suppose that the vertex $[(g_1,u_1)]$ from Fact 1 lies in $P$. Then $[(g_2,u_2)]\in P$ and, in view of  Fact 2, for every
$v \in C$ and each $[(f,v)] \in \esupp_\graphz(X)$, $[(f,v)] \in P$ (because such a vertex cannot be adjacent to both $[(g_1,u_1)]$ and $[(g_2,u_2)]$
in $\graphz$). Consequently,
$Q \subseteq \{[(g,u)] \mid g \in G_A, \, u \in B\}= \{[(1,u)] \mid u \in B\}$. Let $\overline{P}$ and $\overline{Q}$ be the projections of $P$ and $Q$ to $B \cup C$, i.e.,
\begin{align*}
{P}&:=\{v \in B \cup C \mid \mbox{there exists } [(f,v)]\in P \mbox{ for some }f \in G_A\}\subseteq B \cup C \mbox{ and }\\
\overline{Q}&:=\{u \in B \mid [(1,u)]\in Q\}\subseteq B.
\end{align*}
It is easy to see that $\overline{P}$ and $\overline{Q}$ are non-empty and disjoint because $P$ and $Q$ were disjoint, and
$\overline{P} \cup \overline{Q}=B \cup C$ by Fact 2.

Since  $Q\subseteq \link_\graphz(P)$, we see that for all $v \in \overline{P}$ and $u \in \overline{Q}$
there is $f \in G_A$ and  such that $[(f,v)]\in P$ is adjacent to
$[(1,u)]=[(f,u)]\in Q$ in $\graphz$, which implies that $v$ is adjacent to  $u$ in $\ga$. Therefore $u \in \link_\ga(v)$ for all $v \in \overline{P}$ and all $u \in \overline{Q}$,
i.e., $\overline{Q} \subseteq \link_\ga(\overline{P})$.
Recalling that $V\ga=B\sqcup C\sqcup\{a\}=\overline{Q}\sqcup (\overline{P}\cup\{a\}) $ and $\overline{Q}\subseteq B=\link_\ga(a)$, we arrive to a
contradiction with irreducibility of $\ga$.

We have shown that $\esupp_\graphz(X)$ is irreducible and has at least two elements (Fact 1), thus the proposition is proved.
\end{proof}

For a finite subset $X \subseteq G= \graphx \mathfrak{G}$, the \emph{length of $X$ with respect to} $\graphx \mathfrak{G}$ will be defined by
$$|X|_\graphx\coloneqq \sum_{x\in X}|x|_\graphx.$$

\begin{prop}\label{P:comp} Suppose that $\ga$ is a simplicial graph (not necessarily finite) and $G=\GG$ is the graph product of a family
of groups $\G=\{G_v \mid v \in V\ga\}$ with respect to $\ga$.
Let $H$ be a subgroup of $\graphx \mathfrak{G}$ and let $X$ be a finite subset of $H.$
Then there exist a finite graph $\graphz$, a graph product  $\graphz \mathfrak{K}$ over $\graphz$,
and a homomorphism $\varphi\colon H \to \graphz \mathfrak{K}$ 
with the following properties:
\begin{itemize}
	\item each vertex group $K_t$, $t \in V\graphz$, is isomorphic to some $G_v$, $v \in V\graphx$;
	\item  $|\varphi(X)|_\graphz\leq |X|_\graphx$;
	\item $\varphi$ is injective on $\gen{X}$;
	\item $\esupp_\graphz(\varphi(X))=V\graphz$;
	\item $\rho_{\{t\}}(\varphi(H))\neq \{1\}$ in $\graphz \mathfrak{K}$ for all $t\in V\graphz$.
\end{itemize}
Moreover, if $\esupp_\ga(X)$ contains at least two elements and is irreducible in $\ga$ then $\graphz$ will be irreducible and
$|V\graphz|\ge 2$. If, additionally,
 $\rho_{\{a\}}(X)=\{1\}$ in $\graphx \mathfrak{G}$ for some $a\in\esupp_\graphx(X)$, then $|\varphi(X)|_\graphz< |X|_\graphx.$
\end{prop}

\begin{proof} The required graph product can be constructed by applying the following two-step procedure several times:

Step 1. Let $\Theta$ be the full subgraph $\ga_E$ of $\ga$, where $E:=\esupp_\ga(X) \subseteq V\ga$. Since $|X|<\infty$ one sees that $E$ and the graph $\Theta$
are both finite (cf. Remarks \ref{rem:es_supp} and \ref{rem:es_supp-prop_parab}).
Note that the canonical retraction $\rho_E: G \to G_E$
is length-reducing: $|X|_\ga \ge |\rho_E(X)|_\Theta$, where the latter length is measured in the full subgroup
$G_E$ of $G$, with $G_E$ being considered as a graph product of the family $\{G_v \mid v \in E\}$ with respect to $\Theta$.
Moreover, $\rho_E$ is injective on $\gen{X}$ and $\esupp_\Theta(\rho_E(X))=\esupp_\Theta(\rho_E(\gen{X}))=V\Theta=E$ by Lemma \ref{lem:proj_on_esupp}.
Set $H_1:=\rho_E(H) \leqslant G_E$ and $X_1:=\rho_E(X) \subseteq H_1$.

Now, if for all $a\in E=V\Theta$, $\rho_{\{a\}}(H_1)\neq \{1\}$  in $G_E$ then $G_E$ is the required graph product $\graphz\mathfrak{K}$ where
$\varphi:H \to G_E$ is the restriction of $\rho_E$ to $H$. Otherwise, proceed to step 2.

Step 2.
Suppose that there is $a\in V\Theta=\esupp_\Theta(X_1)$ such that $\rho_{\{a\}}(H_1)=\{1\}$  in $G_E$, i.e., $H_1 \subseteq \ker \rho_{\{a\}}$.
Then, by Theorem~\ref{thm:kernel}, $\ker \rho_{\{a\}}$ is a graph product $\graphz \mathfrak{K}$ of groups, each of which is isomorphic to some $G_v$, $v \in V\graphx$.
By our assumptions and Remark \ref{rem:es_supp}, there is some $x_0\in X_1$ such that $a\in \supp_\Theta(x_0)$.
Hence, by Theorem~\ref{thm:kernel}.\eqref{eq:comp}, $|x_0|_\graphz<|x_0|_\Theta$ and  $|x|_\graphz \le |x|_\Theta$ for all $x \in X_1-\{x_0\}$. Consequently
$|X_1|_\graphz<  |X_1|_\Theta\le |X|_\ga.$ 
We also have that $H_1 \leqslant \graphz \mathfrak{K}$ and if $E=\esupp_\ga(X)$ was irreducible in $\ga$ and contained at least two elements, then
$\Theta=\ga_E$ is an irreducible graph with at least two vertices, hence $\esupp_\graphz(X_1)$ will be irreducible in $\graphz$ and
$|\esupp_\graphz(X_1)|\ge 2$ by Proposition \ref{prop:irred_subset}.

Since the non-negative integer $|X|_\ga$ is being decreased with each application of Step~2,
after repeating this two-step procedure finitely many times we will obtain a graph product with all the desired properties.
\end{proof}

The following fact (cf.  \cite[Cor. 3.28]{Green}) can be deduced immediately from the above proposition:

\begin{cor}\label{cor:torfree} Let $p$ be a prime number. If each vertex group $G_v$, $v \in V\ga$, has no $p$-torsion then the graph product $G=\GG$ has no $p$-torsion.
Thus if each vertex group is torsion-free then $G$ is torsion-free.
\end{cor}

\begin{proof} If $G$ contains an element $g$, of order $p$, then take $H\coloneqq \gen{g}$ and $X\coloneqq \{g\}$. Applying Proposition~\ref{P:comp}, we will find $v\in V\ga$ a homomorphism
$\alpha\colon  H \to G_v$ such that $\alpha(H) \neq \{1\}$. Since $H$ is a cyclic group of prime order, the latter implies that $\alpha$ is injective, hence $\alpha(g) \in G_v$ has order $p$.
\end{proof}



\begin{df}
Let (P) be a property of groups (preserved by group isomorphisms) and let $\mathcal{I}$ be a collection of cardinals. We will say that a group is $\mathcal{I}$-\textit{locally} (P), if every
non-trivial $\mathcal{I}$-generated subgroup has (P).
\end{df}

In this work we will mainly be interested in the properties of being indicable and profi.
A group $G$ is said to be \emph{indicable} if it has an infinite cyclic quotient;
$G$ will be called \emph{profi} if it possesses a proper finite index subgroup. Clearly every indicable group is profi, but not vice-versa.

Note that if a group $G$ maps onto an indicable (resp. profi) group, then $G$ itself is indicable (resp. profi).
More generally, we shall say that (P) is a \textit{lifting} property  provided the following holds:
if $G$ is a group possessing a {non-trivial} quotient $Q$ with property (P) then $G$ itself has (P).
Thus the properties of indicability and profi are lifting properties.

\begin{prop}\label{prop:lift}
Suppose that $\mathcal{I}$ is a collection of cardinals, $\graphx$ is a simplicial graph,
$G=\GG$ is the graph product of a family of groups $\G=\{G_v \mid v \in V\ga\}$ with respect to $\ga$.
Let  {\normalfont(P)} be a lifting property of groups. If every vertex group $G_v$, $v \in V\ga$, is  $\mathcal{I}$-locally {\normalfont(P)} then the graph product $G$ is also
$\mathcal{I}$-locally {\normalfont(P)}.
\end{prop}

\begin{proof} Let $H$ be a non-trivial $\mathcal{I}$-generated subgroup of $G$. Pick any $h \in H-\{1\}$ and set $X\coloneqq \{h\}$.
By Proposition~\ref{P:comp} we can find the graph product  $\graphz \mathfrak{K}$ and a homomorphism $\varphi\colon  H \to \graphz \mathfrak{K}$ enjoying all of its claims.
In particular $\varphi(h)\neq 1$, hence $\graphz$ has at least one vertex $t \in V\graphz$. Since $\rho_{\{t\}}(\varphi(H)) \neq \{1\}$ in $K_t$ and $K_t$ is $\mathcal I$-locally (P)
(as it is isomorphic to a vertex group of the original graph product), we can conclude that  $(\rho_{\{t\}}\circ\varphi)(H)$ has (P), and so $H$ has (P) because this property is lifting.
\end{proof}

Proposition \ref{prop:lift} has the following immediate corollary, which will later be used in the proofs of Theorems B and C:

\begin{cor}\label{cor:loc_ind-profi}
Let $\mathcal{I}$ be a collection of cardinals. Then the graph product of groups is $\mathcal{I}$-locally indicable (resp.  $\mathcal{I}$-locally profi)
if and only if each vertex group is $\mathcal{I}$-locally indicable (resp.  $\mathcal{I}$-locally profi).
\end{cor}

A group $G$ is \emph{right orderable} if there a total order invariant under the right action of $G$ on itself by multiplication.
Clearly, subgroups of right orderable groups are right orderable.  By \cite[Thm.~2]{Burns-Hale},  a group is right orderable if and only if every non-trivial finitely generated
subgroup maps onto a non-trivial right orderable group. Using this fact together with an argument, similar the one applied in Proposition~\ref{prop:lift}, one can achieve the following statement:

\begin{cor}\label{cor:orderable} A graph product of right orderable groups is right orderable.
\end{cor}

A different proof of Corollary \ref{cor:orderable} has been independently obtained by I. Chiswell in \cite{Chiswell}, where  it is also proved that a graph product of bi-orderable groups is bi-orderable.



\section{Proofs of Theorems B and C}\label{S:proof}
\begin{lemma}
\label{L:subgraph_irred}
Let $\ga$ be a non-empty finite irreducible graph. Then there is a vertex $v \in V=V\ga$ such that the full subgraph $\ga_A$, of $\ga$, is again irreducible, where $A\coloneqq V-\{v\}$.
\end{lemma}

\begin{proof}
Recall that the graph $\ga$ is irreducible if and only if the complement graph $\ga'$ is connected.
It is not difficult to see that $\ga'$, as any non-empty finite connected graph, has a vertex $v$ that is not a cut vertex.
Then the full subgraph of $\ga'$ spanned by  $A\coloneqq V-\{v\}$ (obtained from $\ga'$ by removing $v$ and all the edges adjacent to it),
is connected. Thus $\ga_A$ is  irreducible.
\end{proof}

Our methods for proving of Theorems B and C are similar, and so we are going to prove both of them in parallel.

\begin{lemma}\label{lem:H_large_mapsonto}
Suppose that $\mathcal{I}$ is a collection of cardinals and $G$ is a group splitting as a free amalgamated product $G=G_1*_{G_3}G_2$, where $G_j \leqslant G$, $j=1,2,3$.
Let $H \leqslant G$ be an $\mathcal{I}$-generated subgroup such that $H$ is not virtually cyclic, $H$ is not contained in a conjugate of $G_j$, $j=1,2$, in $G$,
and $H \cap gG_3g^{-1}=\{1\}$ for all $g \in G$. Then the following are true.
\begin{enumerate}
 \item[(i)] If $\mathcal I$ is ample and $G$ is $\mathcal I$-locally profi, then $H$ is large.
 \item[(ii)] If $G$ is $\mathcal{I}^-$-locally indicable, then $H$ maps onto \ff.
\end{enumerate}
\end{lemma}

\begin{proof} By a generalization of Kurosh's subgroup theorem (see \cite[I.5.5, Thm. 14]{Serre}),
$ H=\left(\mbox{\Large $*$}_{i \in I} H_i\right) * F$ where $F$ is a free group, $I$ is a set of indices (possibly empty),
and for each $i \in I,$ there are $g_i \in G$ and $j=j(i) \in \{1,2\} $ such that $H_i=H \cap g_iG_jg_i^{-1}\neq \{1\}$.
Since $H$ is $\mathcal{I}$-generated, $F$ and every $H_i$ are also $\mathcal I$-generated, as retracts of $H$.

If the rank of $F$ is at least $2$, then $H$ maps onto \ff{} and both (i) and (ii) hold. Thus we can further assume that $F$ is either trivial or is isomorphic to $\Z.$

(i) If $F\cong \Z$, since $H$ is not cyclic we have $I \neq \emptyset$, thus there exists some $i_0\in I$. Since $G$ is $\mathcal I$-locally profi,
$H_{i_0}$ possesses a normal subgroup $N_0$ of finite index $n\geq 2$.
Then $H$ has an epimorphism onto the virtually free group $F_1\coloneqq H_{i_0}/N_0*F$. The above conditions on
$n$ and $F$ imply that $F_1$ contains a \nafs{}, and so it is not virtually cyclic. Thus $F_1$ is large and, hence, so is $H$.

In the case when $F=\{1\}$, $I$ must have at least two elements because $H$ is not contained in a conjugate of $G_j$, $j=1,2$, by assumptions.
If $|I|\ge 3$, choose three distinct indices $i_0, i_1, i_2 \in I$. Since $G$ is $\mathcal I$-locally profi, each $H_{i_l}$ possesses a normal subgroup $N_l$ of
finite index $n_l \ge 2$, $l=0,1,2$.
Then $H$ maps onto the free product $H_{i_0}/N_{0}*H_{i_1}/N_{1}*H_{i_2}/N_{2}$, of three non-trivial finite subgroups, and hence it is large.

If $I= \{i_0,i_1\}$ and $|H_{i_l}|=2$ for $l=0,1$, then $H\cong \D_{\infty},$ contradicting to $H$ not being virtually cyclic.
Thus without loss of generality we can assume that $|H_{i_0}|\ge 3$. Again, since $G$ is $\mathcal I$-locally profi, $H_{i_l}$ possesses a
normal subgroup $N_l$ of finite index $n_l \ge 2$ for $l=0,1$. Moreover, if $n_0=2$ then $N_{i_0} \neq \{1\}$ and it is $\mathcal I$-generated by
Lemma \ref{lem:card_of_fi}, hence $N_{i_0}$ is profi and it must contain a proper finite index subgroup. Therefore we can assume that $n_{i_0} \ge 3$.
Then $H$ maps onto the free product $(H_{i_0}/N_0)*(H_{i_1}/N_1)$ of two finite subgroups, of orders $n_0\geq 3$ and $n_1\geq 2$, and so $H$ is large.
This completes the proof for (i).

(ii) If $F\cong \Z$, since $H$ is not cyclic there exists $i_0\in I$.
By the Grushko-Neumann Theorem  \cite[I.10.6]{Dicks-Dunwoody}, $H_{i_0}$ is $\mathcal{I}^-$-generated because $H$ is $\mathcal{I}$-generated.
Hence $H_{i_0}$ maps onto $\Z$ as $G$ is $\mathcal{I}^-$-locally indicable.
It follows that $H$ maps onto $\Z*F$ which is isomorphic to \ff.

If $F=\{1\}$, the assumptions imply that $|I|\ge 2$, so we can find distinct $i_0,i_1\in I$. Since both $H_{i_0}$ and $H_{i_1}$
are non-trivial and $H$ is $\mathcal{I}$-generated, the Grushko-Neumann Theorem implies that these free factors are $\mathcal{I}^-$-generated.
The $\mathcal{I}^-$-local indicability of $G$ yields that $H_{i_0}$ and $H_{i_1}$ each map onto $\Z$, hence $H$ maps onto $\Z*\Z \cong \mathbb{F}_2$, as required.
\end{proof}

For the rest of this section $\graphx$ will be a finite simplicial graph, $\mathfrak{G}=\{G_v \mid v \in V\graphx\}$ will be a family of
groups and $G=\graphx\mathfrak{G}$ will be the corresponding graph product.
Before proving the main results, we are going to establish two sufficient criteria for a subgroup of a graph product to be large or to map onto \ff.

The following elementary observation will be useful:
\begin{rem}\label{rem:non-tr_proj} If $H \leqslant G$ non-trivially projects to every vertex group then $H$ is not contained in any proper parabolic subgroup of $G$.
\end{rem}

\begin{thm}\label{thm:preBC}
Let $\mathcal{I}$ be a collection of cardinals and let  $\ga$ be a finite  irreducible graph with $|V\ga| \ge 2$. Suppose that
$H \leqslant G=\GG$ is an $\mathcal{I}$-generated subgroup such that $H$ is not virtually cyclic and $\rho_{\{v\}}(H)\neq \{1\}$ for every $v\in V\ga.$
\begin{enumerate}
\item[(a)] If $\mathcal I$ is ample and $G_v$ is $\mathcal I$-locally profi for every $v \in V\ga$, then $H$ is large.
\item[(b)] If $G_v$ is $\mathcal{I}^-$-locally indicable for each $v \in V\ga$, then $H$ maps onto \ff.
\end{enumerate}
\end{thm}

\begin{proof}  If the collection $\mathcal I$ contained no cardinals greater than $1$, we would get a contradiction with our assumptions
because only cyclic subgroups could be $\mathcal{I}$-generated in this case, but $H$ is not virtually cyclic by the assumptions.
Therefore $\mathcal{I}^-$ contains at least one non-zero cardinal, implying that any cyclic subgroup of $G$ is
$\mathcal I^-$-generated. Observe, also, that, by Corollary \ref{cor:loc_ind-profi}, in case (a) $G$ is $\mathcal I$-locally profi and
in case (b) $G$ is $\mathcal I^-$-locally indicable. The proof will proceed by induction on $|V\ga|$.

Base of induction: $|V\ga|=2$. Suppose that $V\ga=\{a,b\}$ and set $A=\{a\}$, $B=\{b\}$. In this case $G=G_A*G_B$ (as $\ga$ is irreducible) and $H$ is not contained in a conjugate of
$G_A$ or $G_B$ by Remark \ref{rem:non-tr_proj}. Therefore all the assumptions of Lemma \ref{lem:H_large_mapsonto} are satisfied, hence
$H$ has the required property.

Step of induction: $|V\ga|\ge 3$ and we can suppose that the required statement has been proved for all $\mathcal{I}$-generated subgroups, of (possibly different)
graph products over finite irreducible graphs with fewer vertices, mapping non-trivially to each vertex group.

By Lemma~\ref{L:subgraph_irred} there is a vertex $v \in V\coloneqq V\ga$ such that the full subgraph $\ga_A$, of $\ga$,
is again irreducible, where $A\coloneqq V-\{v\}$. Let $C=\link_\ga(v) \subset V$ and $B=C\cup\{v\}\subset V$,
then $G$ splits as the amalgamated free product $G_A*_{G_C}G_B$. Note that $|A|\ge 2$ and $C \subsetneqq A$ because  $\ga$ is irreducible,
hence $A$ and $B$ are proper subsets of $V$.

Recall that for the non-empty proper subset $A\subsetneqq V\ga,$ $G_A=\ga_A\G_A$ is a graph product with fewer vertices than $G$,
where $\G_A\coloneqq \{G_v \in \G \mid v \in A\}$, and  $\rho_A(H)$ is $\mathcal{I}$-generated.
Moreover, $\rho_A(H)$ projects non-trivially to each vertex group of $G_A$, because for every $v \in A$ the canonical retraction $\rho_v\colon G \to G_v$
factors through the retraction $\rho_A\colon G\to G_A$ (that is, $\rho_{\{v\}}(H)=(\rho_{\{v\}}\circ\rho_A)(H)$).
Thus if $\rho_A(H)$ is not virtually cyclic in $G_A$,  by induction we know that $\rho_A(H)$ satisfies the required properties
($\rho_A(H)$ is large in case (a) or maps onto \ff{} in  case (b)), and these properties are inherited by $H$.
Hence we can suppose that $\rho_A(H)$ is virtually cyclic. Applying Theorem~\ref{thm:structure} and Remark \ref{rem:non-tr_proj}, we see that
either $\rho_A(H)\cong\Z$ or $\rho_A(H)\cong \D_\infty.$

\textit{Case 1.} If $\rho_A(H) \cong \Z$, then $\rho_A(H)$ has trivial intersection with any proper parabolic subgroup of $G_A$ by Remark \ref{rem:non-tr_proj}
and Lemma  \ref{lem:virt_cyc-parab}.
Take any $g \in G$, set $g_1\coloneqq \rho_A(g) \in G_A$ and, recalling the inclusion $C \subsetneqq A$, observe that
$$gHg^{-1} \cap G_C=\rho_A\left(gHg^{-1} \cap G_C\right) \subseteq g_1 \rho_A(H) g_1^{-1} \cap G_C=\{1\},$$ since $g_1^{-1}G_Cg_1$ is a
proper parabolic subgroup of $G_A$. By Remark \ref{rem:non-tr_proj} and our hypothesis,  $H$ is not contained in a conjugate of
$G_A$ or $G_B$, therefore we can apply Lemma \ref{lem:H_large_mapsonto} to reach the necessary conclusion.

\textit{Case 2.} If $\rho_A(H) \cong \D_\infty$ then $G$ contains $2$-torsion, hence $G_v$ contains $2$-torsion for some $v \in V$, by Corollary \ref{cor:torfree}.
The latter is impossible in case (b) since the cyclic group of order $2$ is $\mathcal{I}^-$-generated but is not indicable. Hence we are in case (a), and we need to show that $H$ is large.

By Remark \ref{rem:non-tr_proj} and Lemma \ref{lem:virt_cyc-parab}, $\rho_A(H)$ has finite intersection  with any proper parabolic subgroup of $G_A$, hence
the infinite cyclic subgroup $N\leqslant \rho_A(H)$, of index $2$, will have trivial intersection  with each proper parabolic subgroup of $G_A$.
Set $H_1\coloneqq \rho_A^{-1}(N) \cap H$ and note that $|H:H_1|=2$, $H_1 \lhd H$ and $\rho_A(H_1)=N$. By the same argument as in Case 1, $gH_1g^{-1}\cap G_C=\{1\}$ for all $g \in G$.
If $H_1$ were contained in a conjugate of $G_A$ or $G_B$, then $H \leqslant \No_G(H_1)$ would be contained in a proper parabolic subgroup of $G$ by Corollary \ref{cor:norm_in_irred},
contradicting  our assumptions together with Remark \ref{rem:non-tr_proj}.
We also need to recall Lemma \ref{lem:card_of_fi}, which claims that $H_1$ is $\mathcal I$-generated.
Thus we see again that all of the conditions of Lemma \ref{lem:H_large_mapsonto}.(i) are satisfied, and so $H_1$ is large. And since $H_1$ has index $2$ in $H$, we are able to
conclude that $H$ is large as well.
\end{proof}

\begin{thm}\label{thm:main_crit}
Let $\mathcal{I}$ be a collection of cardinals and let  $\ga$ be a finite graph. Suppose that
$H \leqslant G=\GG$ is an $\mathcal{I}$-generated subgroup such that $H$ is not virtually cyclic and
$\esupp_\ga(H)$ is an irreducible subset of $V\ga$ with $|\esupp_\ga(H)|\ge 2$.
\begin{enumerate}
\item[(a)] If $\mathcal I$ is ample and $G_v$ is $\mathcal I$-locally profi for every $v \in V\ga$, then $H$ is large.
\item[(b)] If $G_v$ is $\mathcal{I}^-$-locally indicable for each $v \in V\ga$, then $H$ maps onto \ff.
\end{enumerate}
\end{thm}

\begin{proof} According to Lemma \ref{lem:proj_on_esupp}, we can project everything onto $G_E$, where $E:=\esupp_\ga(H)$, and further
assume that $\esupp_\ga(H)=V\ga$, that is $H$ is not contained in any proper parabolic subgroup of $G$.

By Theorem \ref{thm:structure} and Remark \ref{rem:4-core}, $H$ contains a finitely generated subgroup $M$ such that $M$ is not virtually cyclic and
$\esupp_\ga(M)=V\ga$. Choose any finite generating set $X$ of $M$ and apply Proposition \ref{P:comp} to find a finite irreducible
graph $\graphz$ with at least $2$ vertices, a graph product $K=\graphz\mathfrak{K}$, and a homomorphism $\varphi:H \to K$
such that every vertex group $K_t$, $t \in V\graphz$, is isomorphic to some $G_v$, $v \in V\ga$, $\varphi$ is injective on $M$ and
$\varphi(H)$ projects non-trivially to each vertex group of $K$.

Since $\varphi(M) \cong M$ is not virtually cyclic, we can conclude that $\varphi(H)$ is not virtually cyclic. Therefore both
claims of the theorem now follow from Theorem \ref{thm:preBC}.
\end{proof}

We are now in a position to prove Theorems B and C.

\begin{proof}[Proofs of Theorems B and C]
The argument will proceed by induction on $|V\ga|$. Note that the base of our induction, $|V\ga|=1$,
holds because the vertex groups satisfy the corresponding form of Tits alternative.
Therefore we can suppose that
$|V\ga|\geq 2$ and the required statement has been proved for all (possibly different)
graph products over graphs with fewer vertices.

Observe that, in Theorem B, the Strong Tits Alternative for $G_v$ together with (P$5$) imply that
$G_v$ is $\mathcal I$-locally profi for all $v \in V\ga$. Similarly, in Theorem C, every vertex group $G_v$ will be  $\mathcal{I}^-$-locally indicable
by the Strongest Tits Alternative combined with (P$6$).

Let $H$ be an $\mathcal{I}$-generated subgroup of $G.$ By Theorem A we know that if $H$ does not contain a copy of \ff,
then $H$ belongs to $\mathcal{C}.$ Thus we can further suppose that $H$ is not virtually cyclic.

For every  non-empty proper subset $A\subsetneqq V\ga,$ $G_A=\ga_A\G_A$ is a graph product with fewer vertices than $G$,
where $\G_A\coloneqq \{G_v \in \G \mid v \in A\}$, and  $\rho_A(H)$ is $\mathcal{I}$-generated.
Then, by the induction hypothesis we know that either $\rho_A(H)$ is in $\mathcal{C}$ or is large (in Theorem B), or maps onto \ff{} (in Theorem C).
In the latter two cases $H$ inherits the needed properties from $\rho_A(H)$.

Summarizing, we can restrict to the situation when $H$ is not virtually cyclic and for all $A\subsetneqq V\ga,$
$\rho_A(H)$ belongs to $\mathcal{C}$.

If $\ga$ is not irreducible then there exist $A,B\subsetneqq V\ga$ such that $G=G_A\times G_B.$
In this case $H\leqslant \rho_A(H)\times\rho_B(H)$ belongs to $\mathcal{C}$ by (P$1$) and (P$2$). 

Suppose that $\ga$ is irreducible. If $E:=\esupp_\ga(H)\neq V\ga$, then $H\cong \rho_E(H)$ (see Lemma \ref{lem:proj_on_esupp})
belongs to $\mathcal C$. Thus we can assume that $\esupp_\ga(H)= V\ga$,
which enables us to apply Theorem \ref{thm:main_crit} and achieve that either $H$ is large (in Theorem B) or $H$ maps onto \ff{} (in Theorem  C).
\end{proof}

\section{A Rips-type construction and two examples}\label{S:example}
The goal of this section  is to show that the claims of Corollaries  \ref{cor:2-gen} and \ref{cor:RAAG} from the Introduction cannot be improved in an obvious way.
For instance, it is natural to ask, whether, after dropping the torsion-freeness assumption from Corollary \ref{cor:2-gen}, one can still describe the structure of an arbitrary $2$-generated subgroup.
The example below shows that this may not be easy, as such a subgroup can be large with infinite center; in  particular, it will not be virtually free.

\begin{Ex}\label{ex:2-gen_non_virt_cyc} Let $\ga$ be the simple path of length $2$ (with $3$ vertices and $2$ edges), and let $G$ be the graph product with respect to $\ga$,
such that the vertex group in the middle is $\Z$ and the vertex groups at the end-points are $\Z/3\Z$.
Thus $G$ has the presentation $\langle a,b,c \, \|\, ac=ca,bc=cb, a^3=b^3=1 \rangle$. Consider the $2$-generated subgroup $H = \gen{ac,bc} \leqslant G$. It is not difficult
to see that $H$ is isomorphic to the group given by the presentation $\langle x,y \,\|\, x^3=y^3\rangle$ (under the map sending  $x$ to  $ac$ and $y$ to $bc$). It follows that $H$
has infinite cyclic center generated by $c^3$, and the quotient of $H$ by its center is isomorphic to the free product $\gen{a}*\gen{b}\cong \Z/3\Z*\Z/3\Z$, which is
virtually a non-abelian free group.
\end{Ex}

In order to produce the second example we will use the following Rips-type construction, which could be of independent interest.

\begin{prop}\label{prop:Rips_constr} For every finitely generated recursively presented group $P$ there exists a finitely generated group $K$ and a finitely generated normal subgroup $N \lhd K$
such that all of the following hold:
\begin{itemize}
	\item $K/N \cong P$;
	\item $K$ is a subgroup of some finitely generated right angled Artin group $G$;
	\item $N$ is non-abelian, in particular both $N$ and $K$ map onto \ff.
\end{itemize}
\end{prop}

\begin{proof} First, using the celebrated Higman Embedding Theorem \cite{Higman}, we embed $P$ into a finitely presented group $O$. Next we apply \cite[Prop. 3.1]{A-L-M} to embed $O$
into an infinite finitely presented group $R$ such that $R$ has no proper finite index subgroups. Applying the Haglund-Wise modification of
Rips's construction (see \cite[Thm. 10.1]{Haglund-Wise-0}), we find a finitely generated group $H_1$ and a finitely generated normal subgroup $N_1 \lhd H_1$ such that
$H_1/N_1 \cong R$. Moreover, by Theorem 5.9 from \cite{Haglund-Wise-0}, we know that some finite index subgroup $H \leqslant H_1$
is a subgroup of some finitely generated right angled Artin group $G$. Let $\varphi_1\colon H_1 \to R$ be an epimorphism with
$\ker(\varphi_1)=N_1$. Denote by $\varphi\colon H \to R$ the restriction of $\varphi_1$ to $H$; then  $\varphi(H)$ will have finite index in $R$, hence $\varphi(H)=R$ by the construction of $R$.

Since $P \leqslant R$, we can define $K \leqslant H$ by $K:=\varphi^{-1}(P)$.
Note that $N\coloneqq \ker(\varphi)=N_1\cap H$ is finitely generated because it is a finite index subgroup of $N_1$, which is finitely generated, and $K/N \cong P$.
Since $P$ is finitely generated, it follows that $K$ is also finitely generated.

Suppose that $N$ is abelian. It is well-known that finitely generated right angled Artin groups are linear over $\Z$ (for instance, because they can be embedded into finitely generated
right angled Coxeter groups \cite[Cor. 3.6]{HsuWise}, and the standard geometric representation of the latter is a faithful representation by matrices with integer coefficients
\cite[Cor. 5.4]{Humphreys}). Hence $H \leqslant G$ can be considered as a subgroup of ${\rm GL}_n(\Z)$ for some $n \in \mathbb{N}$.
D. Segal proved (see \cite[Thm. 5, p. 61]{Segal}) that every solvable subgroup of ${\rm GL}_n(\Z)$ coincides with the intersection of finite index subgroups containing it
(in other words, it is closed in the profinite topology of that group).
Therefore the abelian subgroup $N$ will equal to an intersection of finite index subgroups of $H$, which is equivalent to saying that $H/N \cong R$ is residually finite.
But this is impossible because $R$ is an infinite group without proper subgroups of finite index. This contradiction shows that $N$ must be non-abelian, and so,
by Corollary  \ref{cor:RAAG}, both $N$ and $K$ possess epimorphisms onto \ff.
\end{proof}

Recall that the \textit{rank} $\rank(H)$, of a group $H$, is the least cardinality $\lambda$ such that
$H$ can be generated by a subset $X \subset H$ with $|X|=\lambda$. In the context of this work it is natural to define the \textit{free rank} $\frank(H)$, of $H$,
as the supremum of all cardinals $\lambda$ such that $H$ has an epimorphism onto a free group of rank $\lambda$. Evidently, $\frank(H)\le \rank(H)$, and in view of
Corollary \ref{cor:RAAG}, it makes sense to ask whether one can find a lower bound for $\frank(H)$ in terms of $\rank(H)$ for
subgroups $H$ of a given finitely generated right angled Artin group $G$. Note that for any non-trivial abelian subgroup $H \leqslant G$,
$\rank(H)\le r\coloneqq \rank(G) \in \N$ (see Corollary \ref{cor:rank_ab_in_RAAG}), and so
$\rank(H) -r+1 \le 1= \frank(H)$. However, the next statement shows that no such lower bound is possible in general.

Below $\omega=|\N|$ denotes the first infinite cardinal; thus $ \N\cup\{\omega\}$ is the set of all non-zero countable cardinals.

\begin{prop}\label{prop:ex_fr_rank} There exists a finitely generated right angled Artin group $G$, a natural number $n \in N$ and a collection of subgroups
$\{H_i \mid i \in \N\cup\{\omega\}\}$ of $G$
such that $\frank(H_i) \le n$ and $i \le \rank(H_i) \le i+n$ for all $i\in  \N\cup\{\omega\}$. In particular,
$\rank(H_\omega)=\omega$, $\rank(H_i)\in \N$ for all $i \in \N$, and $\rank(H_i) \to \infty$ as $i \to \infty$.
\end{prop}

\begin{proof} Let $P:=\Z {\wr} \Z$ be the (restricted) wreath product of two infinite cyclic groups. Then $P$ is $2$-generated, recursively presented and contains a subgroup $Q$
which is isomorphic to the free abelian group $\Z^\omega$, of rank $\omega$.
According to Proposition \ref{prop:Rips_constr}, there is a finitely generated right angled Artin group $G$ and finitely generated subgroups $N,K \leqslant G$
such that $N\lhd K$ and $K/N \cong P$. Let $\varphi:K \to P$ denote an epimorphism with $\ker(\varphi)=N$.

For each $i \in \N$ choose a subgroup $R_i \le Q$ of rank $i$ and let $R_\omega\coloneqq Q$. Now, let $H_i$, $i \in \N\cup\{\omega\}$,
denote the full $\varphi$-preimage of $R_i$ in $K$ and set $n\coloneqq \rank(N) \in \N$.
Clearly $i=\rank(R_i) \le \rank(H_i) \le \rank(R_i)+\rank(N)=i+n$ for all $i \in \N\cup\{\omega\}$.

On the other hand, if $i \in \N\cup\{\omega\}$ and  $\psi$ is an epimorphism from $H_i$ onto some free group $\mathbb{F}$, then $\psi(N)$ is a finitely generated normal
subgroup of $\mathbb{F}$. If $\psi(N)=\{1\}$ then $\psi$  factors through the restriction of $\varphi$ to $H_i$, i.e., $\mathbb{F}$
is a homomorphic image of $\varphi(H_i)=R_i$, which is a free abelian group. Hence $\rank(\mathbb{F})\le 1 \le n$.

If $\psi(N) \neq \{1\}$ then $\psi(N)$ must have finite index in $\mathbb{F}$ (see \cite[I.3.12]{L-S}), hence, recalling Schreier index formula (\cite[I.3.9]{L-S}),
we obtain $\rank(\mathbb{F}) \le \rank(\psi(N))\le \rank(N)=n$.
Therefore $\frank(H_i) \le n$ for all  $i \in \N\cup\{\omega\}$, as claimed.
\end{proof}

\medskip

\noindent{\textbf{{Acknowledgments}}}

The authors are grateful to Martin Bridson, Martin Dunwoody, Ian Leary, Armando Martino and Alexander Ol'shanskii for helpful conversations.
We also thank the referee for his/her suggestions.

\end{document}